\documentclass[11pt,a4paper,reqno]{amsart}%
\usepackage{amssymb}
\usepackage{amsmath}
\usepackage{amsfonts}
\usepackage{amsthm}
\usepackage{color}
\usepackage[usenames,dvipsnames]{xcolor}
\usepackage{bbm}
\usepackage{graphicx}
\usepackage{rotating}
\usepackage{hyperref}
\usepackage{mathrsfs}
\usepackage[bbgreekl]{mathbbol}
\usepackage{cmll}
\usepackage[all, cmtip]{xy}%
\providecommand{\U}[1]{\protect\rule{.1in}{.1in}}
\newtheorem{theorem}{Theorem}[subsection]
\newtheorem{corollary}[theorem]{Corollary}
\newtheorem{lemma}[theorem]{Lemma}
\newtheorem{proposition}[theorem]{Proposition}
\newtheorem*{theorem*}{Theorem}

\theoremstyle{definition}
\newtheorem{definition}[theorem]{Definition}
\newtheorem{examplex}[theorem]{Example}
\newenvironment{example}
  {\pushQED{\qed}\examplex}
  {\popQED\endexamplex}
  
\theoremstyle{remark}
\newtheorem{remarkx}[theorem]{Remark}
\newenvironment{remark}
  {\pushQED{\qed}\remarkx}
  {\popQED\endremarkx}

\thanks{}

\numberwithin{equation}{section}

\makeatletter
\def\@tocline#1#2#3#4#5#6#7{\relax
  \ifnum #1>\c@tocdepth 
  \else
    \par \addpenalty\@secpenalty\addvspace{#2}%
    \begingroup \hyphenpenalty\@M
    \@ifempty{#4}{%
      \@tempdima\csname r@tocindent\number#1\endcsname\relax
    }{%
      \@tempdima#4\relax
    }%
    \parindent\z@ \leftskip#3\relax \advance\leftskip\@tempdima\relax
    \rightskip\@pnumwidth plus4em \parfillskip-\@pnumwidth
    #5\leavevmode\hskip-\@tempdima
      \ifcase #1
       \or\or \hskip 1.5 em \or \hskip 2em \else \hskip 3em \fi%
      #6\nobreak\relax
    \hfill\hbox to\@pnumwidth{\@tocpagenum{#7}}\par
    \nobreak
    \endgroup
  \fi}
\makeatother

\begin{document}
\newcommand{\R}{{\mathbbm R}}
\newcommand{\C}{{\mathbbm C}} 
\newcommand{\T}{{\mathbbm T}}
\newcommand{\D}{{\mathbbm D}}
\renewcommand{\P}{\mathbb P}

\newcommand{\Aa}{{\mathcal A}}
\newcommand{\Ii}{{\mathbbm K}}
\newcommand{\Jj}{{\mathbbm J}}
\newcommand{\Nn}{{\mathcal N}}
\newcommand{\Ll}{{\mathcal L}}
\newcommand{\Tt}{{\mathcal T}}
\newcommand{\Gg}{{\mathcal G}}
\newcommand{\Dd}{{\mathcal D}}
\newcommand{\Cc}{{\mathcal C}}
\newcommand{\Oo}{{\mathcal O}}

\newcommand{\pr}{\operatorname{pr}}
\newcommand{\bla}{\langle \hspace{-2.7pt} \langle}
\newcommand{\bra}{\rangle\hspace{-2.7pt} \rangle}
\newcommand{\blq}{[ \! [}
\newcommand{\brq}{] \! ]}
 \newcommand{\into}{\mathbin{\vrule width1.5ex height.4pt\vrule height1.5ex}}

\title{The Local Structure of Generalized Contact Bundles}

\author{Jonas Schnitzer}
\address{DipMat, Universit\`a degli Studi di Salerno, via Giovanni Paolo II n${}^\circ$ 123, 84084 Fisciano (SA) Italy.}
\email{jschnitzer@unisa.it}

\author{Luca Vitagliano}
\address{DipMat, Universit\`a degli Studi di Salerno, via Giovanni Paolo II n${}^\circ$ 123, 84084 Fisciano (SA) Italy.}
\email{lvitagliano@unisa.it}

\begin{abstract}
Generalized contact bundles are odd dimensional analogues of generalized complex manifolds. They have been introduced recently and very little is known about them. In this paper we study their local structure. Specifically, we prove a local splitting theorem similar to those appearing in Poisson geometry. In particular, in a neighborhood of a regular point, a generalized contact bundle is either the product of a contact and a complex manifold or the product of a symplectic manifold and a manifold equipped with an integrable complex structure on the gauge algebroid of the trivial line bundle.
\end{abstract}
\maketitle

\tableofcontents

\section*{Introduction}

Generalized complex manifolds have been introduced by Hitchin in \cite{H2003} and further investigated by Gualtieri in \cite{G2011}, and the literature about them is now rather wide. Generalized complex manifolds are necessarily even dimensional and they encompass symplectic and complex manifolds as extreme cases. A natural question is what is the odd dimensional analogue of a generalized complex manifold. Several answers to this question appeared already in the literature but the works on generalized geometry in odd dimensions are still sporadic \cite{IW2005,V2008,PW2011,W2012, S2015,AG2015}. Recently, A.~Wade and the second author proposed a partially new definition of an odd dimensional analogue of a generalized complex manifold, called a \emph{generalized contact bundle} \cite{VW2016}. Generalized contact bundles are a slight generalization of Iglesias-Wade integrable generalized almost contact structures \cite{IW2005} to the realm of (generically non-trivial) line bundles, and encompass not necessarily coorientable contact manifolds as an extreme case. At the other extreme they encompass line bundles equipped with an integrable complex structure on their gauge algebroid. In turn, such line bundles are intrinsic models for so called \emph{normal almost contact manifolds} \cite{B2002}. In our opinion, generalized contact bundles have an advantage over previous proposals of a generalized contact geometry: they have a firm conceptual basis in the so called \emph{homogenization scheme} \cite{VW2017}, which is, in essence, a dictionary from contact and related geometries to symplectic and related geometries. In principle, applying the dictionary is straightforward: it is enough to replace functions on a manifold $M$ with sections of a line bundle $L \to M$, vector fields over $M$ with derivations of $L$, etc. In practice, applying the dictionary can be actually challenging, and may lead to interesting new features \cite{G2013, V2015a, V2015, LOTV2014, LTV2016, BT2016, T2017a, T2017, VW2016b, VW2017}.

In \cite{VW2016} the authors define generalized contact bundles, and study their structure equations, showing, in particular, that every generalized contact bundle is a Jacobi bundle
\cite{Kirillov1976, Lichn1978, Marle1991}. This puts odd dimensional generalized geometry in the framework of Jacobi geometry. In this paper we begin a systematic study of generalized contact bundles by studying their local structure. Our main results are two splitting theorems. In this introduction we provide for them rough statements to be better explained and made precise in the bulk of the paper.

\begin{theorem*}[\textbf{A}]
Let $M$ be a manifold equipped with a generalized contact bundle, and let $x_0 \in M$ be a point in an odd dimensional characteristic leaf of $M$. Then, locally around $x_0$, $M$ is isomorphic, up to a $B$-field transformation, to the product of a contact manifold and a homogeneous generalized complex manifold whose homogeneous Poisson structure vanishes at a point.
\end{theorem*}

\begin{theorem*}[\textbf{B}]
Let $M$ be a manifold equipped with a generalized contact bundle, and let $x_0 \in M$ be a point in an even dimensional characteristic leaf of $M$. Then, locally around $x_0$, $M$ is isomorphic, up to a $B$-field transformation, to the product of a symplectic manifold and a manifold with a generalized contact bundle whose Jacobi structure vanishes at a point.
\end{theorem*}

We also explicitly discuss the local structure of a generalized contact bundle in a neighborhood of a regular point, proving the following two local normal form theorems.

\begin{theorem*}[\textbf{C}]
Let $M$ be a $(2n + 2d + 1)$-dimensional manifold equipped with a generalized contact bundle, and let $x_0 \in M$ be a point in a $(2d +1)$-dimensional characteristic leaf of $M$. If $x_0$ is a regular point, then, locally around $x_0$, $M$ is isomorphic, up to a $B$-field transformation, to the product of the standard $(2d +1)$-dimensional contact manifold $(\mathbbm R^{2d+1}, \theta_{\mathit{can}})$ and the standard complex space $\mathbbm C^n$.
\end{theorem*}

\begin{theorem*}[\textbf{D}]
Let $M$ be a $(2n + 2d + 1)$-dimensional manifold equipped with a generalized contact bundle, and let $x_0 \in M$ be a point in a $2d$-dimensional characteristic leaf of $M$. If $x_0$ is a regular point, then, locally around $x_0$, $M$ is isomorphic, up to a $B$-field transformation, to the product of the standard $2d$-dimensional symplectic space $(\mathbbm R^{2d}, \Omega_{\mathit{can}})$ and the cylinder $\mathbbm R \times \mathbbm C^n$ equipped with the canonical complex structure on the gauge algebroid of the trivial line bundle $(\mathbbm R \times \mathbbm C^n) \times \mathbbm R \to \mathbbm R \times \mathbbm C^n$.
\end{theorem*}

We stress that the word ``product'' in the statements of Theorems (A)-(D) does not exactly refer to the (cartesian) product of manifolds, but rather to a certain technical notion of \emph{product of line bundles} over possibly different base manifolds that we explain in Subsection \ref{subsec:prod_DJ} below. Here we only mention that, if $L_1 \to M_1$ and $L_2 \to M_2$ are line bundles, then a \emph{product} of them in our sense is a line bundle over $M_1 \times M_2$ equipped with additional structure (see diagram (\ref{diag:product}). We also mention here that 

the proof of Theorem (D) requires proving that certain Dolbeault-like cohomologies associated with a complex structure on the gauge algebroid of a line bundle are locally trivial.

For our proofs, we use an adaptation of the methods of Bursztyn-Lima-Meinrenken \cite{BLM2016}. This suggests that the Splitting Theorems (A) and (B) are a manifestation of more general \emph{normal form theorems} around appropriate transversals to a characteristic leaf (see \cite{BLM2016} for more details). This is indeed the case, as showed by the first author in \cite{S20XX}. Even more, it is showed in \cite{S20XX} that there are normal forms of Jacobi (and, more generally, Dirac-Jacobi) structures around appropriate transversals, extending the splitting theorems of Dazord-Lichnerowicz-Marle \cite{DLM1991}, and paralleling similar normal forms in Poisson (and Dirac) Geometry (see \cite{BLM2016}, see also \cite{FM2017}).  In this paper, we prefer to stick on a less general formulation. Our choice is mainly dictated by space reasons: we think the paper contains already enough material, including several generalities with a certain interest, besides our main results. We actually hope that this paper could also serve as one possible reference for (local aspects of) Generalized Contact Geometry in future works.
%

The paper is organized as follows. In Section \ref{Sec:preliminaries} we collect the necessary preliminaries on gauge algebroids and Jacobi structures. In Section \ref{Sec:gen_cont_DJ} we recall the notions of generalized contact bundles \cite{VW2016} and complex Dirac-Jacobi structures \cite{V2015}. In this section we also discuss in details symmetries of the omni-Lie algebroid, which plays for generalized contact and Dirac-Jacobi bundles a similar role as the generalized tangent bundle plays for generalized complex and Dirac manifolds. Finally we discuss \emph{homogeneous generalized complex structures}, and a suitable notion of \emph{product of Dirac-Jacobi bundles}, which appears to be unavoidable in a precise formulation of our splitting theorems. In Section \ref{Sec:transv} we describe in details the structures induced on the characteristic leaves of a generalized contact structure, and on their transversals. Section \ref{Sec:splitting} contains our main results: the splitting theorems around a point in a contact and around a point in a locally conformal symplectic leaf. In the last Section \ref{Sec:regular} we prove, as corollaries, local normal form theorems around a regular point. Finally, in Appendix \ref{appendix}, we discuss a very special class of generalized contact structures: complex structures on the gauge algebroid of a line bundle. We prove a local normal form theorem analogous to the Newlander-Nirenberg theorem, and the local vanishing of an associated Dolbeault-like cohomology. Both are consequences of their standard even dimensional counterparts. 

We assume the reader is familiar with the fundamentals of Lie algebroids, Dirac manifolds and generalized complex structures.

\section{Preliminaries}\label{Sec:preliminaries}
\subsection{The gauge algebroid}\label{subsec:gauge_alg}
A derivation of a vector bundle $E \to M$ is an $\mathbbm R$-linear operator
\[
\Delta : \Gamma (E) \to \Gamma (E)
\]
satisfying the following \emph{Leibniz rule}
\[
\Delta (f \varepsilon) = X(f) \varepsilon + f \Delta \varepsilon, \quad f \in C^\infty (M), \quad \varepsilon \in \Gamma (E),
\]
for a, necessarily unique, vector field $X \in \mathfrak X (M)$, called the \emph{symbol} of $\Delta$ and denoted by $\sigma (\Delta)$. Derivations are sections of a Lie algebroid $DE \rightarrow M$, called the \emph{gauge algebroid} of $E$, whose anchor is the symbol, and whose bracket is the commutator of derivations \cite{M2005}. The fiber $D_x E$ of $DE$ over a point $x \in M$ consists of $\mathbbm R$-linear maps $\Delta : \Gamma (E) \to E_x$ satisfying the Leibniz rule $\Delta (f \varepsilon) = v(f) \varepsilon_x + f(x) \Delta \varepsilon$ for some tangent vector $v \in T_x M$: the \emph{symbol} of $\Delta$.

The correspondence $E \mapsto DE$ is functorial, in the following sense. Let $F \to N$ and $E \to M$ be two vector bundles, and let $\Phi : F \to E$ be a \emph{regular} vector bundle map, i.e.~a bundle map, covering a smooth map $\phi : N \to M$, which is an isomorphism on fibers. Then $\Phi$ induces a (generically non-regular) vector bundle map $D \Phi : D F \to D E$ via formula
\[
D \Phi (\Delta) \varepsilon = (\Phi \circ \Delta) ( \Phi^\ast \varepsilon),
\]
for all $\Delta \in D F$, and $\varepsilon \in \Gamma (E)$. Here $\Phi^\ast \varepsilon$ is the \emph{pull-back} of $\varepsilon$ along $\Phi$, i.e.~it is the section of $F$ given by $(\Phi^\ast \varepsilon)_y = \Phi|_{F_y}^{-1} (\varepsilon_{\phi(y)})$, $y \in N$. The vector bundle map $D \Phi$ will be sometimes denoted by $\Phi_\ast$ if there is no risk of confusion. Correspondence $\Phi \mapsto D \Phi$ preserves identity and compositions.

Derivations of a vector bundle $E$ can be seen as \emph{linear vector fields} on $E$, i.e.~vector fields generating a flow by vector bundle automorphisms. Namely, for every derivation $\Delta$ of $E$, there exists a unique flow $\{ \Phi_t \}$ by vector bundle automorphisms $\Phi_t : E \to E$ such that 
\[
\Delta \varepsilon =  \frac{d}{dt} |_{t = 0} \Phi_t^\ast \varepsilon
\]
for all $\varepsilon \in \Gamma (E)$.

The gauge algebroid acts tautologically on the vector bundle $E$. Accordingly, there is a \emph{de Rham complex} of $DE$ with coefficients in $E$, denoted $(\Omega^\bullet_E, d_D)$. Cochains in $(\Omega^\bullet_E, d_D)$ will be referred to as \emph{Atiyah forms}. They are vector bundle maps $\wedge^\bullet DE \to E$. The differential $d_D$ is given by the usual formula. Atiyah forms can be pulled-back along regular vector bundle maps. Namely, let $F \to N$ and $E \to M$ be vector bundle, and let $\Phi : F \to E$ be a regular vector bundle map covering a smooth map $\phi : M \to N$. For $\omega \in \Omega^k_E$, we define $\Phi^\ast \omega \in \Omega^k_F$ via
\[
(\Phi^\ast \omega)_{y} (\Delta_1, \ldots, \Delta_k) = \Phi|_{F_y}^{-1} \circ \omega_{\phi (y)} (\Phi_\ast \Delta_1, \ldots, \Phi_\ast \Delta_k)
\]
for all $y \in N$, and $\Delta_1, \ldots, \Delta_k \in D_{y} F$. One can also take the Lie derivative $\mathcal L_\Delta := [\iota_\Delta, d_D]$ of Atiyah forms along a derivation $\Delta$, and all these operators satisfy the usual 
\emph{Cartan calculus identities}. Additionally, there is a distinguished derivation, namely the identical one: $\mathbb 1 : \Gamma (E) \to \Gamma (E)$, $\varepsilon \mapsto \varepsilon$, and the contraction $\iota_{\mathbb 1}$ of Atiyah forms with $\mathbb 1$ is a contracting homotopy for $(\Omega^\bullet_E, d_D)$. In particular, $(\Omega^\bullet_E, d_D)$ is acyclic.

 In the case of  a line bundle $L \to M$, every first order differential operator $\Gamma (L) \to \Gamma (L)$ is a derivation. Consequently, there are vector bundle isomorphisms $DL \cong \mathsf{Hom} (J^1 L, L)$, and $J^1 L \cong \mathsf{Hom} (D L, L)$, and a non-degenerate pairing $\langle -, - \rangle : J^1 L \otimes DL  \to L$, where $J^1 L \to M$ is the first jet bundle of $L$. In this case, the identical derivation $\mathbb 1$ spans the kernel of the symbol and there is a short exact sequence:
 \begin{equation}\label{ses:Spencer}
 0 \longrightarrow \mathbbm R_M \longrightarrow DL \overset{\sigma}{\longrightarrow} TM \longrightarrow 0,
 \end{equation}
 where $\mathbbm R_M = M \times \mathbbm R$ is the trivial line bundle over $M$. Dually, there is a short exact sequence
  \begin{equation}\label{ses:Spencer*}
 0 \longleftarrow  L \longleftarrow J^1 L \longleftarrow T^\ast M \otimes L \longleftarrow 0.
 \end{equation}
The embedding $T^\ast M \otimes L \hookrightarrow J^1 L$ extends to an embedding
\[
\Omega^\bullet (M, L) \hookrightarrow \Omega^\bullet_L,
\] 
of $L$-valued forms on $M$ into Atiyah forms on $L$, consisting in composing with the symbol $\sigma : DL \to TM$. We will often interpret $\Omega^\bullet (M, L)$ as a subspace in $\Omega^\bullet_L$ without further comments. Notice that $\Omega^1_L = \Gamma (J^1 L)$, and the first differential $d_D : \Gamma (L) \to \Omega^1_L$ agrees with the first jet prolongation $j^1 : \Gamma (L) \to \Gamma (J^1 L)$.

 \begin{remark}[Atiyah forms on the trivial line bundle]\label{rem:trivial}
 When $L = \mathbbm R_M$ is the trivial line bundle, then sections of $L$ are just functions on $M$, both sequences (\ref{ses:Spencer}) and (\ref{ses:Spencer*}) splits canonically via the standard flat connection in $\mathbbm R_M$, and we have 
 \[
 \begin{aligned}
 D \mathbbm R_M & = TM \oplus \mathbbm R_M, \\
 J^1 \mathbbm R_M & = T^\ast M \oplus \mathbbm R_M.
 \end{aligned}
 \] 
 In this case, a generic derivation is of the form $X + f$ where $X$ is a vector field and $f$ is a function. Similarly a generic section of $J^1 \mathbbm R_M$ is of the form $\eta + g \cdot j^1 1$, where $\eta$ is a $1$-form,  $g$ is a function, and $j^1 1$ is the first jet prolongation of the constant function $1 \in C^\infty (M)$. In the following, we will denote $\mathfrak j = j^1 1$. Then we have
 \[
 j^1 f = df + f \cdot \mathfrak j, \quad f \in C^\infty (M).
 \]
 More generally, any Atiyah form $\omega \in \Omega^\bullet_{\mathbbm R_M}$ can be uniquely written as
 \begin{equation}\label{eq:decomposition}
 \omega = \omega_0 + \omega_1 \wedge \mathfrak j,
 \end{equation}
 with $\omega_0, \omega_1 \in \Omega^\bullet (M)$, and we used the symbol to give Atiyah forms the structure of a graded $\Omega^\bullet (M)$-module. The correspondence $\omega \mapsto (\omega_0, \omega_1)$ establishes an isomorphism of graded $\Omega^\bullet (M)$-modules:
\[
\Omega^\bullet_{\mathbbm R_M} \cong \Omega^\bullet (M) \oplus \Omega^{\bullet -1} (M).
\]
In terms of the decomposition (\ref{eq:decomposition}) the natural operations on Atiyah forms read as follows:
\[
\begin{aligned}
d_D \omega & = d \omega_0 + \left(d\omega_1 + (-)^{|\omega_0|} \omega_0\right) \wedge \mathfrak j \\
\iota_{X + f} \omega & = \iota_X \omega_0 + (-)^{|\omega_1|} f \omega_1 + \iota_X \omega_1 \wedge \mathfrak j \\
\mathcal L_{X + f} \omega& = \mathcal L_X \omega_0 + f \omega_0 + \omega_1 \wedge df + \left(\mathcal L_X \omega_1 + f \omega_1\right) \wedge \mathfrak j.
\end{aligned}
\]
for all $X + f \in \Gamma (D \mathbbm R_M)$, where the bars ``$|-|$'' denote the degree.
 \end{remark}

\subsection{Jacobi bundles and their characteristic foliations}

Jacobi manifolds were introduced by Kirillov \cite{Kirillov1976}, and independently, Lichnerowicz \cite{Lichn1978}, as generalizations of Poisson manifolds. Here we adopt to Jacobi manifolds the 
slightly more intrinsic approach via Jacobi bundles \cite{Marle1991} (see also \cite{T2017}). Jacobi bundles encompass (not necessarily coorientable) contact manifolds as non-degenerate instances.

Let $L \to M$ be a line bundle. A \emph{Jacobi structure} on $L$ is a Lie bracket
\[
\{-,-\} : \Gamma (L)\times \Gamma (L) \to \Gamma (L)
\]
which is also a first order bi-differential operator or, equivalently, a bi-derivation. The bracket $\{-,-\}$ is
also called the \emph{Jacobi bracket}. A \emph{Jacobi bundle} is a line bundle equipped with a
Jacobi structure. A Jacobi bracket $\{-,-\}$ can be regarded as a $2$-form
\[
J : \wedge^2 J^1 L \to L
\] 
satisfying an additional integrability condition, and in the following we will often take this point of view.

\begin{example}\label{ex:Jac_nd}
Every contact manifold is canonically equipped with a Jacobi bundle containing a full information on the 
contact structure. Indeed, let $(M, H)$ be a contact manifold, i.e.~$H \subset TM$ is a maximally non-integrable hyperplane distribution, and consider the normal line bundle $L := TM / H$. The distribution $H$ 
can be equivalently encoded in a line bundle valued $1$-form $\theta \in \Omega^1 (M, L)$: the canonical 
projection $\theta : TM \to L$. In its turn $\theta$ can be seen as an Atiyah $1$-form on $L$. One can prove 
that $\omega := d_D \theta \in \Omega^2_L$ is a non-degenerate (and closed) Atiyah $2$-form (see, e.g.,
 \cite{V2015}). Here, the non-degeneracy means that the induced vector bundle map
 \[
 \omega_\flat : DL \to J^1 L
 \]
 is invertible. Its inverse $\omega_\flat^{-1}$ is the sharp map $J^\sharp  : J^1 L \to DL$ of a (unique, non-degenerate) Jacobi structure $J := \omega^{-1} : \wedge^2 J^1 L \to L$. Conversely, every non-degenerate Jacobi structure on a line bundle $L \to M$, determines 
 a contact structure $H \subset TM$ on $M$, with $TM/H = L$. For some more details, see the discussion
at the beginning of Subsection \ref{Sec:Transv_contact}.
\end{example}

\begin{example}
Every locally conformal symplectic (lcs) manifold is canonically equipped with a Jacobi bundle containing a full information on the lcs structure. We adopt a slightly more intrinsic approach to lcs manifolds. Namely, in this paper, a lcs structure on a line bundle $L \to M$ is a pair $(\Omega, \nabla)$, where $\nabla$ is a flat connection in $L$, and $\Omega$ is an $L$-valued $2$-form on $M$, which is 1) non-degenerate and 2) closed with respect to the connection differential $d^\nabla : \Omega^\bullet (M, L) \to \Omega^{\bullet +1 } (M, L)$. When $L = \mathbbm R_M$ is the trivial line bundle we recover the usual definition. So let $L \to M$ be a line bundle equipped with an lcs symplectic structure $(\nabla, \Omega)$. The bracket
\[
\{ -, -\} : \Gamma (L) \times \Gamma (L) \to \Gamma (L), \quad (\lambda, \mu) \mapsto \Omega^{-1}(d^\nabla \lambda, d^\nabla \mu)
\]
is a Jacobi bracket. If we interpret it as a $2$-form $J : \wedge^2 J^1 L \to L$, it is easy to see that the rank of $J$ is $\dim M$. Conversely, every Jacobi structure $J$ on a line bundle $L \to M$ such that $\operatorname{rank} J = \dim M$ determines an lcs structure on $L$. For some more details, see the discussion at the beginning of Subsection \ref{Sec:Transv_lcs}.
\end{example} 

Similarly as a Poisson manifold, a manifold $M$ equipped with a Jacobi bundle $(L, J)$ possesses a canonical (generically singular) foliation, called the characteristic foliation and defined as follows. Consider the sharp map associated to $J$, $J^\sharp : J^1 L \to DL$. Composing with the symbol we get a map $\sigma J^\sharp : J^1 L \to TM$, whose image is an involutive distribution on $M$. The integral foliation $\mathcal F$ of $\operatorname{im} \sigma J^\sharp$ is the \emph{characteristic foliation} of the Jacobi bundle $(L, J)$, and its leaves are the \emph{characteristic leaves}. Odd dimensional leaves of $\mathcal F$ are naturally contact manifolds, while even dimensional leaves are lcs manifolds. For more details about properties of the characteristic leaves in Jacobi geometry see, e.g., \cite{T2017} (see also Section \ref{Sec:transv}).

When $L = \mathbbm R_{M} \to M$ is the trivial line bundle, a Jacobi bracket
  $\{-,-\}$ on $L$ is equivalent to a \emph{Jacobi pair}, i.e.~a pair $(\Lambda, E)$,
 consisting of a bivector $\Lambda \in \mathfrak X^2 (M)$ and a vector field 
 $E \in \mathfrak X (M)$ such that
\[
[\Lambda, \Lambda]^{SN} = 2 E \wedge \Lambda, \quad \text{and} \quad [E, \Lambda]^{SN} = 0,
\]
where $[-,-]^{SN}$ is the \emph{Schouten-Nijenhuis bracket} of multivectors. The equivalence is provided by the following formula:
\[
\{ f, g \} = \Lambda (f, g) + E(f) g - f E(g), \quad f,g \in C^\infty (M).
\]

\begin{example}\label{ex:J_can}
On $\mathbbm R^{2d +1}$, with coordinates $(x^1, \ldots, x^d,  p_1, \ldots, p_d, u)$, there is a canonical Jacobi pair $(\Lambda_{\mathit{can}}, E_{\mathit{can}})$ given by
\[
\Lambda_{\mathit{can}} := \frac{\partial}{\partial p_i} \wedge \left( \frac{\partial}{\partial x^i} + p_i \frac{\partial}{\partial u} \right), \quad \text{and} \quad E_{\mathit{can}} = \frac{\partial}{\partial u}.
\]
We denote by $J_{\mathit{can}}$ the Jacobi structure corresponding to the Jacobi pair $(\Lambda_{\mathit{can}}, E_{\mathit{can}})$.
\end{example}

\section{Generalized contact and Dirac-Jacobi geometry}\label{Sec:gen_cont_DJ}

\subsection{The omni-Lie algebroid and its symmetries}

The natural arena for \emph{generalized geometry in odd dimensions} is the 
\emph{omni-Lie algebroid} $\mathbbm D L$ of a line bundle $L \to M$ \cite{CL2010}.
Recall that $\mathbbm D L = DL \oplus J^1 L$, where $DL \to M$ is the gauge algebroid. The omni-Lie algebroid possesses the following structures:
\begin{itemize}
\item[$\triangleright$] a natural projection
\begin{equation}\label{eq:pr}
\pr_{D} : \D L \longrightarrow D L;
\end{equation}
\item[$\triangleright$] a non-degenerate, symmetric, split signature $L$-valued $2$-form 
\[
\bla -, - \bra : \D L \otimes \D L \to L
\]
 given by:
\[
\bla (\Delta , \psi), (\square, \chi) \bra := \langle \chi ,\Delta \rangle + \langle \psi, \square \rangle;
\]
\item[$\triangleright$] a (non-skew symmetric, Dorfman-like) bracket 
\[
\blq -,- \brq : \Gamma (\D L) \times \Gamma (\D L) \to \Gamma (\D L)
\]
 given by:
\begin{equation}\label{eq:Dorf_brack}
\blq (\Delta, \psi), (\square, \chi) \brq := \left([\Delta, \square], \Ll_\Delta \chi - \iota_\square d_{D} \psi \right)
\end{equation}
\end{itemize}
for all $\Delta, \square \in D L$, and all $\psi, \chi \in \Gamma (J^1 L)$. These structures satisfy certain identities that we do not report here (for more details see, e.g., \cite{V2015}). Most of them are just the obvious analogues of those holding for the \emph{standard Courant algebroid}: the \emph{generalized tangent bundle} $\mathbbm T M = TM \oplus T^\ast M$. Accordingly, the rest of this subsection is just an adaptation from similar features of $\mathbbm T M$.

We now describe symmetries of the omni-Lie algebroid $\mathbbm D L$. First of all, for a vector bundle $E \to M$, we denote by $\mathrm{Aut} (E)$ the group of its automorphisms.
\begin{definition}
A \emph{Courant-Jacobi automorphism} of $\mathbbm D L$ is a pair $(\mathbb \Phi, \Phi)$
consisting of
\begin{enumerate}
\item an automorphism $\mathbb \Phi$ of the vector bundle $\mathbbm D L$, and
\item an automorphism $\Phi$ of $L$,
\end{enumerate}
such that $\mathbb \Phi$ and $\Phi$ cover the same diffeomorphism $\phi : M \to M$, and, additionally,
\[
\begin{aligned}
D\Phi \circ p_D  & = p_D \circ \mathbb \Phi  \\
 \Phi^\ast \bla \alpha, \beta \bra & = \bla \mathbb \Phi^\ast \alpha, \mathbb \Phi^\ast \beta \bra  \\
 \mathbb \Phi^\ast \blq \alpha, \beta \brq & = \blq \mathbb \Phi^\ast \alpha, \mathbb \Phi^\ast \beta \brq
\end{aligned}
\]
for all $\alpha, \beta \in \Gamma (\mathbbm D L)$. The group of Courant-Jacobi automorphisms is denoted by $\mathrm{Aut}_{CJ} (\mathbbm D L)$.
\end{definition}

\begin{example}
Let $B$ be a closed Atiyah $2$-form, i.e.~$B \in \Omega^2_L$ and $d_D B = 0$ (in particular, $B$ is exact). Denote by $e^B : \mathbbm DL \to \mathbbm DL$ the vector bundle automorphism defined by 
\[
e^B (\Delta, \psi) := (\Delta, \psi + \iota_\Delta B), \quad (\Delta, \psi) \in \mathbbm D L.
\]
Using the decomposition $\mathbbm D L = DL \oplus J^1 L$ we can write $e^B$ in matrix form:
\begin{equation}
e^B = \left(
\begin{array}{cc}
\mathrm{id} & 0 \\
B_\flat & \mathrm{id}
\end{array}
\right).
\end{equation}
An easy computation shows that $(e^B, \mathrm{id})$ is a Courant-Jacobi automorphism. We will refer to it as a \emph{$B$-field transformation}, adopting the same terminology as for Courant automorphisms of the generalized tangent bundle. Clearly $e^0 = \mathrm{id}$, $e^{B_1} \circ e^{B_2} = e^{B_1 + B_2}$, and $(e^B)^{-1} = e^{-B}$, for all closed Atiyah $2$-forms $B, B_1, B_2$, showing that $B$-field transformations form a(n abelian) subgroup of $\mathrm{Aut}_{CJ} (\mathbbm D L)$ isomorphic to $Z^2_L$: the group of $2$-cocycles in $(\Omega^\bullet_L, d_L)$.
\end{example}

\begin{example}
Let $\Phi : L \to L$ be a vector bundle automorphism covering a diffeomorphism $\phi : M \to M$. Define a vector bundle automorphism $\mathbbm D \Phi : \mathbbm D L \to \mathbbm D L$ via
\[
\mathbbm D \Phi (\Delta, \psi) := (\Phi_\ast (\Delta), (\Phi^{-1})^\ast \psi), \quad (\Delta, \psi) \in \mathbbm D L.
\]
It is easy to see that $(\mathbbm D \Phi, \Phi)$ is a Courant-Jacobi automorphism. Additionally $\mathbbm D \mathrm{id} = \mathrm{id}$, $\mathbbm D \Phi_1 \circ \mathbbm D \Phi_2 = \mathbbm D (\Phi_1 \circ \Phi_2)$, and $(\mathbbm D \Phi)^{-1} = \mathbbm D (\Phi^{-1})$, for all $\Phi, \Phi_1, \Phi_2 \in \mathrm{Aut} (L)$, showing that Courant-Jacobi automorphisms of the form $\mathbbm D \Phi$ form a subgroup of $\mathrm{Aut}_{CJ} \mathbbm D L$ isomorphic to $\mathrm{Aut} (L)$. Finally, let $B \in Z^2_L$ and $\Phi \in \mathrm{Aut}(L)$. Then
\[
   e^{ B} \circ \mathbbm D \Phi = \mathbbm D \Phi\circ  e^{\Phi^\ast B} .
\]
In particular we see that $B$-field transformations and automorphisms of $L$ generate a subgroup in $\mathrm{Aut}_{CJ} (\mathbbm D L)$ isomorphic to
\[
Z^2_L \rtimes \mathrm{Aut} (L)
\]
where $\mathrm{Aut} (L)$ acts on $Z^2_L$ (from the right) via pull-backs.
\end{example}

Actually, exactly as for the generalized tangent bundle, $B$-field transformations and automorphisms of $L$ generate the full group of Courant-Jacobi automorphisms, according to the following proposition which we report here for completeness.

\begin{proposition}\label{prop:Aut_CJ}
Let $L \to M$ be a line bundle. Then
\[
\mathrm{Aut}_{CJ}(\mathbbm D L) \cong Z^2_L \rtimes \mathrm{Aut} (L).
\]
\end{proposition}

\begin{proof}
The proof follows exactly the same lines as in standard Dirac geometry (see e.g.~\cite[Proposition 2.5]{G2011}) and we omit it.
\end{proof}

We now pass to infinitesimal symmetries of $\mathbbm D L$. First of all, for a vector bundle $E \to M$, denote by $\mathfrak{aut} (E)$ the Lie algebra of its infinitesimal automorphisms. As already remarked, $\mathfrak{aut} (E)$ is canonically isomorphic to the Lie algebra $\Gamma (DE)$ of derivations $E$.

\begin{definition}
An \emph{infinitesimal Courant-Jacobi automorphism} of $\mathbbm D L$ is a pair $(\mathbb \Delta, \Delta)$
consisting of
\begin{enumerate}
\item a derivation $\mathbb \Delta$ of $\mathbbm D L$, and
\item a derivation $\Delta$ of $L$,
\end{enumerate}
such that $\mathbb \Delta$ and $\Delta$ have the same symbol, and, additionally
\[
\begin{aligned}
{}[ \Delta, p_D \alpha]  & = p_D (\mathbb \Delta \alpha) \\
 \Delta \bla \alpha, \beta \bra & = \bla \mathbb \Delta \alpha, \beta \bra + \bla \alpha, \mathbb \Delta \beta \bra \\
 \mathbb \Delta \blq \alpha, \beta \brq & = \blq \mathbb \Delta \alpha, \beta \brq + \blq \alpha, \mathbb \Delta \beta \brq 
\end{aligned}
\]
for all $\alpha, \beta \in \Gamma (\mathbbm D L)$. Equivalently $(\mathbb \Delta, \Delta)$ generates a flow by Courant-Jacobi automorphisms of $\mathbbm D L$. The Lie algebra of infinitesimal Courant-Jacobi automorphisms is denoted by $\mathfrak{aut}_{CJ} (\mathbbm D L)$.
\end{definition}

\begin{example}
Let $B$ be a closed Atiyah $2$-form, i.e.~$B \in Z^2_L$. Denote by $\overline B $ the endomorphism of $\mathbbm D L$ given by 
\begin{equation}
\overline B :=  \left(
\begin{array}{cc}
 0 & 0 \\
B_\flat & 0
\end{array}
\right).
\end{equation}
Then $(\overline B, 0)$ is an infinitesimal Courant-Jacobi automorphism, exponentiating to the $B$-field transformation corresponding to $B$.
\end{example}

\begin{example}
Let $\square$ be a derivation of $L$. Define a derivation $\mathcal L_\square$ of $\mathbbm D L$ via
\[
\mathcal L_\square (\Delta, \psi) := ([\square, \Delta], \mathcal L_\square \psi), \quad (\Delta, \psi) \in \Gamma (\mathbbm D L).
\]
It is easy to see that $(\mathcal L_\square, \square)$ is an infinitesimal Courant-Jacobi automorphism. Infinitesimal automorphisms of the form $(\mathcal L_\square, \square)$ together with those of the form $(\overline B, 0)$ from the previous example generate a Lie subalgebra in $\mathfrak{aut}_{CJ} (\mathbbm D L)$ isomorphic to
\[
\mathfrak{aut} (L) \ltimes Z^2_L
\]
in the obvious way. Here $\mathfrak{aut} (L)$ acts on $Z^2_L$ via Lie derivatives.
\end{example}

\begin{proposition}
Let $L \to M$ be a line bundle. Then
\[
\mathfrak{aut}_{CJ}(\mathbbm D L) \cong  \mathfrak{aut} (L) \ltimes Z^2_L 
\]
\end{proposition}

\begin{proof}
The proof is similar to that of Proposition \ref{prop:Aut_CJ}, and it is left to the reader.
\end{proof}

\begin{remark}\label{rem:flow}
Let $(B, \square) \in \mathfrak{aut} (L) \ltimes Z^2_L$, and let $(\mathcal L_\square + \overline B, \square)$ be the corresponding infinitesimal Courant-Jacobi automorphism. If $\square$ generates the flow $\{ \Phi_t \}$ by vector bundle automorphism of $L$, then $(\mathcal L_\square + \overline B, \square)$ generates the flow
\[
\{ (e^{C_t} \circ \mathbbm D \Phi_t, \Phi_t) \}
\]
by Courant-Jacobi automorphisms corresponding to $\{(C_t, \Phi_t)\} \subset Z^2_L \rtimes \mathrm{Aut} (L)$
where 
\[
C_t := - \int_0^t (\Phi_{-\epsilon}^\ast B) d \epsilon.
\]
\end{remark}

\subsection{Generalized contact and Dirac-Jacobi bundles}
\subsubsection{Generalized contact bundles}
A \emph{generalized contact bundle} \cite{VW2016} is a line bundle $L \to M$ equipped with a \emph{generalized contact structure}, i.e.~and endomorphism $\mathbbm K : \mathbbm D L \to \mathbbm D L$ of the omni-Lie algebroid 
such that
\begin{itemize}
\item[$\triangleright$] $\mathbbm K$ is \emph{almost complex}, i.e.~$\mathbbm K^2 = - \mathrm{id}$,
\item[$\triangleright$] $\mathbbm K$ is \emph{skewsymmetric}, i.e.~
$
\bla \mathbbm K \alpha, \beta \bra + \bla \alpha,  \mathbbm K \beta \bra = 0
$ 
for all $\alpha, \beta \in \mathbbm D L$, and
\item[$\triangleright$] $\mathbbm K$ is \emph{integrable}, i.e.
$
\blq \mathbbm K \alpha, \mathbbm K \beta \brq - \blq \alpha, \beta \brq - \mathbbm K \blq \mathbbm K \alpha, \beta \brq - \mathbbm K \blq \alpha, \mathbbm K \beta \brq = 0
$
for all $\alpha, \beta \in \Gamma (\mathbbm D L)$.
\end{itemize}
Let $(L \to M, \mathbbm K)$ be a generalized contact bundle. Then $M$ is odd-dimensional. Actually, generalized contact bundles are odd-dimensional analogues of generalized complex manifolds and they encompass contact manifolds and complex structures on the gauge algebroid of $L$ as extreme cases. To see this, use the direct sum decomposition $\mathbbm D L = D L \oplus J^1 L$ to present $\mathbbm K$ in the form
\begin{equation}\label{eq:split}
\Ii = \left(
\begin{array}{cc}
\varphi & J^\sharp \\
\omega_\flat & -\varphi^\dag
\end{array}
\right).
\end{equation}
Then
\begin{itemize}
\item[$\triangleright$] $\varphi : D L \to D L$ is a vector bundle endomorphism, 

\item[$\triangleright$] $\varphi^\dag : J^1 L \to J^1 L$ is its \emph{adjoint}, i.e.~$\langle \varphi^\dag \psi, \Delta \rangle = \langle \psi, \varphi \Delta \rangle$, $(\Delta, \psi) \in \mathbbm D L$,

\item[$\triangleright$] $J : \wedge^2 J^1 L \to L$ is a $2$-form with sharp map $J^\sharp : J^1 L \to D L$, and
\item[$\triangleright$]  $\omega : \wedge^2 D L \to L$ is an Atiyah $2$-form with flat map $\omega_\flat : D L \to J^1 L$.
\end{itemize}
Additionally, $\varphi, J, \omega$ satisfy some identities \cite{VW2016}. In particular, $J$ is a \emph{Jacobi bracket}, so that $(L, J)$ is a \emph{Jacobi bundle} \cite{Marle1991}. When $\varphi = 0$, then $\omega_\flat^{-1} = - J^\sharp$ and $J$ is the Jacobi bracket of a (unique) contact structure $H \subset M$ such that $TM/H = L$ (see Example \ref{ex:Jac_nd}).
When $J = \omega = 0$, then $\varphi$ is an integrable complex structure on the gauge algebroid $DL$ (see Appendix \ref{appendix}).

\begin{remark}
Let $\mathbbm K$ be a generalized contact structure on $L$, and let $(\mathbb \Phi, \Phi)$ be a Courant-Jacobi automorphism of $\mathbbm D L$. Then $\mathbb \Phi \circ \mathbbm K \circ \mathbb \Phi^{-1}$ is a generalized contact structure as well. In particular, for $(\mathbb \Phi, \Phi) = (e^B , \mathrm{id})$, the $B$-field transformation corresponding to a closed Atiyah $2$-form $B$, we obtain that $e^B \circ \mathbbm K \circ e^{-B}$ is a generalized contact structure. The latter will be denoted by $\mathbbm K^{B}$.
\end{remark}

\subsubsection{Dirac-Jacobi bundles}
Similarly as for generalized complex structures, generalized contact structures can be seen as (particularly nice) complex \emph{Dirac-Jacobi structures}, i.e.~complex Dirac structures in the omni-Lie algebroid. A Dirac-Jacobi structure on $L$ \cite{V2015} (see also \cite{CL2010, CLS2011}) is a vector subbundle $\mathfrak L \subset \mathbbm D L$ such that
\begin{itemize}
\item[$\triangleright$] $\mathfrak L$ is \emph{maximally isotropic} with respect to $\bla -, - \bra$;
\item[$\triangleright$] $\mathfrak L$ is \emph{involutive}, i.e.~$\blq \Gamma (\mathfrak L), \Gamma (\mathfrak L) \brq \subset \Gamma (\mathfrak L)$.
\end{itemize}

\begin{remark}
Recall that there is an alternative formulation of (standard) Dirac structures (in particular, generalized complex structures) via the Clifford algebra of the generalized tangent bundle and pure spinors. There is a similar formulation for Dirac-Jacobi bundles, but it is much more involved. Actually, it exploits the technology of \emph{quadratic modules} and their (\emph{even}) \emph{Clifford algebras} (see \cite{BK1994}, see also \cite{A2011} and references therein). We mean to develop this line of thoughts in a future project. 
\end{remark}

\begin{example}\label{ex:DJ}
\quad
\begin{itemize}
\item[$\triangleright$]  Let $L \to M$ be a line bundle and let $J : \wedge^2 J^1 L \to L$ be a bi-differential operator on $\Gamma (L)$. Then $\operatorname{\mathsf{graph}} J := \{(J^\sharp \psi, \psi) : \psi \in J^1 L \}\subset \mathbbm D L$ is a maximally isotropic subbundle, and it is a Dirac-Jacobi structure if and only if $J$ is a Jacobi structure.
\item[$\triangleright$] Let $L \to M$ be a line bundle and let $\omega \in \Omega^2_L$ be an Atiyah $2$-form on $L$. Then $\operatorname{\mathsf{graph}} \omega := \{(\Delta, \iota_\Delta \omega) : \Delta \in DL \}\subset \mathbbm D L$ is a maximally isotropic subbundle, and it is a Dirac-Jacobi structure if and only if $d_D \omega = 0$.
\end{itemize}
\end{example}

Now, let $\mathbbm K$ be a generalized contact structure on $L$. Consider the complexified omni-Lie algebroid $\mathbbm D L \otimes \mathbbm C$, and let $\mathfrak L_{\mathbbm K} \subset \mathbbm DL \otimes \mathbbm C$ be the $+\mathrm{i}$-eigenbundle of $\mathbbm K$. Then $\mathfrak L_{\mathbbm K}$ is a (complex) Dirac-Jacobi structure such that $\mathfrak L_{\mathbbm K} \cap \overline{\mathfrak L}_{\mathbbm K} = 0$, in particular $\mathbbm D L \otimes \mathbbm C= \mathfrak L_{\mathbbm K} \oplus \overline{\mathfrak L}_{\mathbbm K}$. Additionally, $\mathfrak L_{\mathbbm K}$ contains the full information about $\mathbbm K$. Finally, all Dirac-Jacobi structures $\mathfrak L \subset \mathbbm D L \otimes \mathbbm C$ such that $\mathfrak L \cap \overline{\mathfrak L}  = 0$ arise in this way, and we will call them complex Dirac-Jacobi structures of \emph{generalized contact type}.

\begin{remark}
Let $\mathfrak L$ be a Dirac-Jacobi structure on $L$, 
and let $(\mathbb \Phi, \Phi)$ be a Courant-Jacobi 
automorphism of $\mathbbm D L$. Then $\mathbb \Phi (\mathfrak L)$ 
is a Dirac-Jacobi structure as well. In particular, $e^B (\mathfrak L)$ 
is a Dirac-Jacobi structure, denoted by $\mathfrak L^B$, 
for every closed Atiyah $2$-form $B$. If $\mathfrak L = \mathfrak L_{\mathbbm K}$ 
is the $+\mathrm{i}$-eigenbundle of
a generalized contact structure $\mathbbm K$, then 
$e^B (\mathfrak L) = \mathfrak L_{\mathbbm K}^B =\mathfrak L_{\mathbbm K^{B}}$. We stress that, in general, $\mathbbm K^B = e^B \circ \mathbbm K \circ e^{-B}$ is not a honest generalized contact structure, unless $B$ is a \emph{real} Atiyah form (see, e.g., Remark \ref{rem:can_odd}).
\end{remark}

\begin{lemma}\label{lem:p_DLcap}
Let $\mathbbm K$ be a generalized contact structure as in (\ref{eq:split}), and let $\mathfrak L = \mathfrak L_{\mathbbm K} \subset \mathbbm DL \otimes \mathbbm C$ be its $+\mathrm{i}$-eigenbundle. Then
\[
p_D \mathfrak L \cap p_D \overline{\mathfrak L} = \operatorname{im} J^\sharp \otimes \mathbbm C.
\]
\end{lemma}

\begin{proof}
The proof follows the same lines as in standard Generalized Gomplex Geometry \cite[Proposition 3.24 and Corollary 3.25]{G2011}.
\end{proof}

\subsubsection{The $2$-form of a complex Dirac-Jacobi structure}

Let $L \to M$ be a line bundle and let $\mathfrak L \subset \mathbbm D L\otimes \mathbbm C$ be a complex Dirac-Jacobi structure on $L$.
There is a canonical skew-symmetric, $L \otimes \mathbbm C$-valued bilinear map $\varpi$ defined pointwise on the smooth, but not necessarily regular, subbundle $p_D \mathfrak L$ as follows:
\begin{equation}\label{eq:varpi}
\varpi : \wedge^2 p_D \mathfrak L \to L \otimes \mathbbm C, \quad (\Delta, \nabla) \mapsto \langle  \psi, \nabla \rangle,
\end{equation}
here $\psi \in J^1 L \otimes \mathbbm C$ is any $1$-jet such that $(\Delta, \psi) \in \mathfrak L$. It immediately follows from the definition of $\varpi$, that
\begin{equation}\label{eq:L_pi}
\mathfrak L = \left\{ (\Delta, \psi) \in \mathbbm D L \otimes \mathbbm C:  \langle \psi , \nabla \rangle =  \varpi (\Delta, \nabla) \text{ for all }\nabla \in p_D \mathfrak L \right\}.
\end{equation}
Similarly as in generalized complex geometry \cite{AB2006}, when $\mathfrak L$ is of generalized contact type, we can relate $\varpi$ to the corresponding generalized contact structure $\mathbbm K$. First consider the complex conjugate form
\[
\overline \varpi : \wedge^2 p_D \overline{\mathfrak L} \to L \otimes \mathbbm C, \quad (\Delta, \nabla) \mapsto \overline{\varpi (\overline \Delta, \overline \nabla)}.
\]
The real and imaginary parts of $\varpi$:
\[
\operatorname{Re} \varpi := \frac{1}{2}(\varpi + \overline \varpi) \quad \text{and} \quad
\operatorname{Im} \varpi := \frac{1}{2\mathrm{i}}(\varpi - \overline \varpi)
\]
are only defined on the intersection $p_D \mathfrak L \cap p_D \overline{\mathfrak L} = \operatorname{im} J^\sharp \otimes \mathbbm C$, and we have the following 
\begin{lemma}\label{lem:impi}
Let $\mathbbm K$ be a generalized contact structure as in (\ref{eq:split}), let $\mathfrak L = \mathfrak L_{\mathbbm K}$ be its $+\mathrm{i}$-eigenbundle, and let $\varpi$ be the canonical $2$-form on $p_D \mathfrak L$. Then
\begin{equation}\label{eq:impi}
- J ( \psi, \psi') = \operatorname{Im} \varpi (J^\sharp \psi, J^\sharp \psi')
\end{equation}
for all $\psi, \psi ' \in J^1 L \otimes \mathbbm C$.

\end{lemma} 

\begin{proof}
The proof follows the same lines as in Generalized Complex Geometry  (see, e.g.~\cite[Lemmas 3.1, 3.2 and 3.3]{AB2006}).
\end{proof}

\begin{remark}\label{rem:impi}
As recalled in Example \ref{ex:DJ}, every Jacobi structure $J$ on a line bundle $L \to M$ determines a real Dirac-Jacobi structure: $\mathfrak L_J := \operatorname{\mathsf{graph}} J \subset DL \oplus J^1 L = \mathbbm D L$, and we have $p_D \mathfrak L_J = \operatorname{im} J^\sharp$. In turn, for every real Dirac-Jacobi structure $\mathfrak L$ on $L \to M$, there is a canonical $L$-valued $2$-form $ \wedge^2 p_D \mathfrak L \to L$, defined by the same formula (\ref{eq:varpi}) as $\varpi$.  Formula (\ref{eq:impi}) then states that the imaginary part of $\varpi$ agrees with the (complexification of the) $2$-form $\omega_J : \wedge^2 \operatorname{im} J^\sharp \to L$ induced by the Jacobi structure $J$ underlying the generalized contact structure $\mathbbm K$:
\[
\operatorname{Im} \varpi = \omega_J.
\]
Notice that $\omega_J$, hence $\operatorname{Im} \varpi$, is (pointwise) non-degenerate.
\end{remark}

\subsubsection{Backward images of (complex) Dirac-Jacobi structures}
Let $(L \to M, \mathbbm K)$ be a generalized contact bundle, and let $\mathfrak L = \mathfrak L_{\mathbbm K} \subset \mathbbm D L\otimes \mathbbm C$ be its $+\mathrm{i}$-eigenbundle. Like in generalized complex geometry, not all submanifolds of $M$ inherit from $\mathbbm K$ a generalized contact bundle structure. However, all submanifolds of $M$ inherits from $\mathfrak L$ a complex Dirac-Jacobi structure (up to regularity issues), via the \emph{backward image} construction which we now recall for later use. We describe backward images for real Dirac-Jacobi structures. The following considerations extend straightforwardly to complex Dirac-Jacobi structures. So let $L \to M$ be a line bundle equipped with a Dirac-Jacobi structure $\mathfrak L \subset \mathbbm D L$. Consider another line bundle $L_N \to N$ together with a regular vector bundle map $\Phi : L_N \to L$ covering a smooth map $\phi : N \to M$.

We recall from Subsection \ref{subsec:gauge_alg} that, by a \emph{regular} vector bundle map, we mean a fiber-wise invertible vector bundle map. Now, define a subbundle $\Phi^! \mathfrak L \subset \mathbbm D L_N$ via
\[
\Phi^! \mathfrak L := \left\{(\Delta, \Phi^\ast \psi) \in \mathbbm D L_N: (\Phi_\ast \Delta, \psi) \in \mathfrak L \right\}.
\]
The bundle $\Phi^! \mathfrak L$ is always maximal isotropic, hence it has constant rank, but it needs not to be smooth. Nonetheless, if it is smooth, it is a honest vector subbundle, and a Dirac-Jacobi structure on $L_N$, called the \emph{backward image} of $\mathfrak L$ along $\Phi$. There is a simple sufficient condition for smoothness, sometimes referred to as the \emph{clean intersection condition} \cite{B2013b, V2015}. For the purposes of this paper, we only need to know that:

\begin{itemize}
\item[$\triangleright$] if $\phi$ is a submersion, the clean intersection condition is automatically satisfied, hence $\Phi^! \mathfrak L$ is a Dirac-Jacobi structure;
\item[$\triangleright$] if $\phi$ is the immersion of a(n immersed) submanifold $S \hookrightarrow M$, the clean intersection condition boils down to 
\begin{equation}\label{eq:cic}
\operatorname{rank} \left(p_D \mathfrak L|_S + D (L|_S) \right) = \text{constant}.
\end{equation}
\end{itemize}
We refer to \cite{B2013b, V2015} for more details.

\subsubsection{Products of Dirac-Jacobi structures}\label{subsec:prod_DJ}

Let $(M_1, \mathfrak L_1)$ and $(M_2, \mathfrak L_2)$ be manifolds equipped with (standard) Dirac structures, i.e.~maximally isotropic, and involutive subbundles $\mathfrak L_i$ of the \emph{generalized tangent bundles} $\mathbbm T M_i = TM_i \oplus T^\ast M_i$, $i = 1,2$. Then $\mathfrak L_1 \times \mathfrak L_2 \subset \mathbbm T M_1 \times \mathbbm T M_2 = \mathbbm T (M_1 \times M_2)$ is a Dirac structure on the product $M_1 \times M_2$, called the \emph{product of $\mathfrak L_1$ and $\mathfrak L_2$}. It is not immediately obvious how to extend this simple construction to line bundles and Dirac-Jacobi structures. In this section we propose such an extension. The \emph{splitting theorems} of Section \ref{Sec:splitting} will be formulated in terms of the \emph{product of Dirac-Jacobi structures} as defined here. 

Begin with two Dirac-Jacobi structures $\mathfrak L_1, \mathfrak L_2 \subset \mathbbm D L$ on the same line bundle $L \to M$. Let $\mathfrak L_1 \star \mathfrak L_2 \subset \mathbbm D L$ be the (not necessarily regular) subbundle defined by
\begin{equation}\label{eq:star}
\mathfrak L_1 \star \mathfrak L_2 := \left\{ (\Delta, \psi_1 + \psi_2) : (\Delta, \psi_i) \in \mathfrak L_i, \, i = 1,2 \right\}.
\end{equation}
A similar construction for Dirac structures appeared probably in \cite{G2011} for the first time (see also \cite{ABM2009}).
\begin{lemma}\label{Lem: Product}
If  $\mathfrak L_1 \star \mathfrak L_2 \subset \mathbbm D L$ is a smooth subbundle, then it is a Dirac-Jacobi structure.
\end{lemma}

\begin{proof} The lemma can be proven, e.g., by adapting the analogous proof for Dirac structures in \cite{M2016} in the obvious way.

\end{proof}

\begin{remark}\label{rem:star_neutral}
Actually, $B$-field transformations are special instances of the construction (\ref{eq:star}). Namely, let $B$ be a closed Atiyah $2$-form. Then 
\[
\mathfrak L_B = \{ (\Delta,\iota_\Delta B)
 :  \Delta\in DL\} \subset \mathbbm D L 
\]
is a Dirac-Jacobi structure, and, for every other Dirac-Jacobi structure $\mathfrak L$, we have
\[
\mathfrak L \star \mathfrak L_B = \mathfrak L^B.
\]
In particular, the graph $DL \subset \mathbbm D L$ of the null Atiyah $2$-form acts as an identity with respect to the product $\star$.
\end{remark}

We are now ready to define a notion of \emph{product of Dirac-Jacobi structures}. So let $L_i \to M_i$ be line bundles equipped with Dirac-Jacobi structures $\mathfrak L_i \subset \mathbbm D L_i$, $i= 1,2$. We assume we have an additional datum, namely a line bundle $L \to M_1 \times M_2$ over the product, together with regular vector bundle maps $P_i : L \to L_i$ covering the canonical projections $p_i : M_1 \times M_2 \to M_i$, $i = 1,2$:
\begin{equation}\label{diag:product}
\begin{array}{c}
\xymatrix{ & L \ar[ld]_-{P_1} \ar[d] \ar[rd]^-{P_2} &  \\
L_1 \ar[d] & M_1 \times M_2  \ar[ld]_-{p_1}  \ar[rd]^-{p_2} & L_2 \ar[d] \\
M_1 & & M_2
}
\end{array}.
\end{equation}
In this situation we can consider back-ward images, $P_1^! \mathfrak L_1, P_2^! \mathfrak L_2 \subset \mathbbm D L$, and they are regular because the $p_i$ are submersions, $i = 1, 2$. Finally consider
\[
P_1^! \mathfrak L_1 \star P_2^! \mathfrak L_2.
\]
If it is regular, it is a Dirac-Jacobi structure on $L$. 

Now notice that, in view of diagram (\ref{diag:product}), $L$ comes with (partial) connections $D_i$, along $\ker p_i$, $i = 1, 2$, and we can define a genuine connection $D^\times$ in $L$, by putting
\[
D^\times_{X_1 + X_2} \lambda = (D_1)_{X_1} \lambda + (D_2)_{X_2}\lambda, \quad X_i \in \ker p_i,\quad  i = 1,2.
\]

\begin{lemma}\label{lem:product}
The following conditions are equivalent:
\begin{enumerate}
\item The connection $D^\times$ is flat.
\item Around every point of $M_1 \times M_2$ there is a nowhere vanishing local section $\lambda \in \Gamma (L)$ such that $\lambda = P_1^\ast \lambda_1 = P_2^\ast \lambda_2$ for some local sections $\lambda_i \in \Gamma (L_i)$, $i = 1,2$.
\item For every $(\bar x_1, \bar x_2) \in M_1 \times M_2$, there are local trivializations $L_i \cong \mathbbm R_{M_i}$, around $\bar x_i$, $i = 1,2$, and $L \cong \mathbbm R_{M_1 \times M_2}$, around $(\bar x_1, \bar x_2)$, such that the $P_i : L \to L_i$ identify with the projections $\mathbbm R_{M_1 \times M_2} \to \mathbbm R_{M_i}$, $(x_1, x_2; r) \mapsto (x_i ; r)$, where $(x_1, x_2) \in M_1 \times M_2$, and $r \in \mathbbm R$.
\end{enumerate}
\end{lemma}

\begin{proof} \quad

(1) $\Longrightarrow$ (2). Choose as $\lambda$ a nowhere vanishing flat section with respect to $D^\times$.

(2) $\Longrightarrow$ (3). Choose the (local) trivializations $L \cong \mathbbm R_{M_1 \times M_2}$, and  $L_i \cong \mathbbm R_{M_i}$, that identify $\lambda$, and $\lambda_i$, with the constant functions $1$, $i = 1,2$.

(3) $\Longrightarrow$ (1). Obvious.

\end{proof}

When one, hence all three, of the conditions in Lemma \ref{lem:product} hold, we say that the product (\ref{diag:product}) is \emph{flat}. If, additionally, $P_1^! \mathfrak L_1 \star P_2^! \mathfrak L_2$ is regular, we call it the (\emph{flat}) \emph{product} of $\mathfrak L_1$ and $\mathfrak L_2$ (with respect to $P_1, P_2$) and denote it by
\[
\mathfrak L_1 \times^! \mathfrak L_2.
\]
We will provide examples of products of Dirac-Jacobi structures later on. For now we only remark that, if we apply an analogous construction to a pair of Dirac structures, we get exactly their standard product.

\begin{remark}
The above discussion applies to complex Dirac-Jacobi structures without modifications.
\end{remark}

\begin{remark}\label{rem:product_B-field}
Let $\mathfrak L_i$ be Dirac Jacobi structures on the line bundles $L_i \to M_i$, $i = 1,2$, and let $\mathfrak L_1 \times^! \mathfrak L_2$ be a flat product of them with respect to some projections $P_1, P_2$ as in diagram (\ref{diag:product}). Finally, let $B$ be a closed Atiyah $2$-form on $L_1$. It is easy to see that
\[
(\mathfrak L_1 \times^! \mathfrak L_2)^{P_1^\ast B} = \mathfrak L_1^B \times^! \mathfrak L_2.
\]
\end{remark}

\begin{remark}
By now, it should be clear that, when working with Dirac-Jacobi structures, we are working in the category of line bundles and regular vector bundle maps between them. So, no surprise that the appropriate notion of \emph{product} in this setting includes a line bundle on a product manifold, and regular vector bundle maps onto the factors. Nonetheless, in what follows, as we are only interested in local properties, we will use Lemma \ref{lem:product}, and we will mainly consider the case when the line bundles $L_i \to M_i$ and $L \to M_1 \times M_2$ are trivial and the projections $P_i : L \to L_i$ are the obvious ones, $i = 1,2$.
\end{remark}

\subsubsection{Homogeneous generalized complex structures}

As already mentioned, unlike Poisson manifolds, manifolds $M$ with a Jacobi bundle $(L \to M, J)$ possess two kinds of characteristic leaves. Odd dimensional ones inherit from $J$ a canonical contact structure, and we call them \emph{contact leaves}. Even dimensional leaves inherit from $J$ an lcs structure, and we call them \emph{lcs leaves}. Let $\mathcal O$ be a leaf and $x_0 \in \mathcal O$. By a \emph{transversal} to $\mathcal O$ at $x_0$, we mean a submanifold $N$ such that $x_0 \in N$, and $T_{x_0} M = T_{x_0} N \oplus T_{x_0} \mathcal O$. It turns out that transversals to lcs leaves, with the restricted line bundle, possess a canonical Jacobi structure around $x_0$. Additionally, this Jacobi structure vanishes at $x_0$. On the other hand, transversals to contact leaves, possess a canonical homogeneous Poisson structure (up to the choice of a nowhere vanishing section of $L$) around $x_0$. The homogeneous Poisson structure vanishes at $x_0$. Recall that a \emph{homogeneous Poisson structure} on a manifold $M$ is a pair $(\pi, Z)$ where $\pi$ is a Poisson bi-vector, and $Z$ is a vector field, called the \emph{homogeneity vector field}, such that $\mathcal L_Z \pi = - \pi$. 

\begin{example}\label{ex:can_hP}
On $\mathbbm R^{2d}$, with coordinates $(x^1, \ldots, x^d, p_1, \ldots, p_d)$, there is a canonical homogeneous Poisson structure $(\pi_{\mathit{can}}, Z_{\mathit{can}})$ given by
\[
\pi_{\mathit{can}} = \frac{\partial}{\partial p_i} \wedge \frac{\partial}{\partial x^i}, \quad \text{and} \quad Z_{\mathit{can}} = p_i \frac{\partial}{\partial p_i}.
\]
\end{example}

The theory of Jacobi structures is strongly related to that of homogeneous Poisson structures, as the example of transversals to contact leaves shows (see also \cite{DLM1991}). In a similar way generalized contact geometry is strongly related to the theory of \emph{homogeneous generalized complex structures} which we define now. Let $M$ be a manifold.

\begin{definition}\label{def:hom_gen_compl}
A \emph{homogeneous generalized complex structure} on $M$ is a pair $(\mathbbm J, \mathbbm Z)$, where 
\[
\mathbbm J =
\left(
\begin{array}{cc}
A & \pi^\sharp \\
\sigma_\flat & -A^\ast
\end{array}
\right) \in \mathrm{End} (\mathbbm T M)
\]
is a generalized complex structure, and $\mathbbm Z = (Z, \zeta)$ is a section of the generalized tangent bundle $\mathbbm TM$ such that
\begin{itemize}
\item[$\triangleright$] $\mathcal L_Z A = \pi^\sharp \circ (d \zeta)_\flat$,
\item[$\triangleright$] $\mathcal L_Z \pi = - \pi$ (in particular $(\pi, Z)$ is a homogeneous Poisson structure),
\item[$\triangleright$] $\mathcal L_Z \sigma = \sigma - \iota_A d \zeta$,
\end{itemize}
where $\iota_A d \zeta$ is the $2$-form defined by 
\[
(\iota_A d \zeta)(X, Y) = d \zeta (A X, Y) + d \zeta (X, A Y),
\]
for all $X, Y \in \mathfrak X (M)$.
\end{definition}

\begin{remark}
The main motivation for this definition is that the transversal to a contact leaf in the base of a generalized contact bundle is a homogeneous generalized complex manifold, as we will show in Section \ref{Sec:Transv_contact}. 

Another motivation is that a generalized contact structure on a line bundle $L \to M$ is equivalent to a homogeneous generalized complex structure on the frame bundle $\widetilde M = L^\ast \smallsetminus 0$ of $L$. In that case, the section $\mathbbm Z$ is of the special form $(\mathcal E, 0)$, where $\mathcal E$ is the Euler vector field on $\widetilde M$ \cite[Remark 3.6 in the arXiv version]{VW2016}.

The second motivation suggests the following problem: Let $(\mathbbm J, 
\mathbbm Z)$ be a homogeneous generalized complex structure on $M$ 
with $\mathbbm Z = (Z, \zeta)$. \emph{Is it possible to find a closed $2$-form
 $B$ on $M$, equivalently, a $B$-field transformation, such that $(\mathbbm
  J^B, (Z, 0))$ is a homogeneous generalized complex structure?} Answering
   this question goes beyond the scopes of the present paper.
\end{remark}

Definition \ref{def:hom_gen_compl} can be rephrased in terms of the complex Dirac structure associated to $\mathbbm J$, i.e.~the $+\mathrm{i}$-eigenbundle $\mathfrak L_{\mathbbm J}$ of $\mathbbm J$ in the complexified generalized tangent bundle $\mathbbm T M \otimes \mathbbm C$. Namely, we have the following

\begin{proposition}\label{prop:hom_gc}
Let 
\[
\mathbbm J = \left(
\begin{array}{cc}
A & \pi^\sharp \\
\sigma_\flat & -A^\ast
\end{array}
\right)
\] be a generalized complex structure on $M$, and let $\mathbbm Z = (Z, \zeta)$ be a section of the generalized tangent bundle $\mathbbm T M$. Then the following conditions are equivalent.
\begin{enumerate}
\item $(\mathbbm J, \mathbbm Z)$ is a homogeneous generalized complex structure;
\item $([Z, X], \mathcal L_Z \eta - \eta + \iota_X d \zeta) \in \Gamma (\mathfrak L_{\mathbbm J})$ for all $(X, \eta) \in \Gamma (\mathfrak L_{\mathbbm J})$;
\item $([Z, X] + X, \mathcal L_Z \eta + \iota_X d \zeta) \in \Gamma (\mathfrak L_{\mathbbm J})$ for all $(X, \eta) \in \Gamma (\mathfrak L_{\mathbbm J})$.
\end{enumerate}
\end{proposition}

\begin{proof}
It is clear that (2) and (3) are equivalent. It remains to prove that (1) $\Leftrightarrow$ (3). Assume first that $(\mathbbm J, \mathbbm Z)$ is a homogeneous generalized complex structure, let $\alpha = (X, \eta) \in \Gamma (\mathfrak L_{\mathbbm J})$ and compute
\begin{equation}\label{eq:comp}
\begin{aligned}
\mathbbm J \left( 
\begin{array}{c}
[Z, X] + X \\
\mathcal L_Z \eta + \iota_X d \zeta
\end{array}
 \right) 
 & = \left(
\begin{array}{cc}
A & \pi^\sharp \\
\sigma_\flat & -A^\ast
\end{array}
\right) \left( 
\begin{array}{c}
[Z, X] + X \\
\mathcal L_Z \eta + \iota_X d \zeta
\end{array}
 \right) \\
 & =
 \left( 
\begin{array}{c}
A[Z, X] + AX  + \pi^\sharp \mathcal L_Z \eta + \pi^\sharp \iota_X d \zeta \\
\sigma_\flat [Z, X] + \sigma_\flat X - A^\ast \mathcal L_Z \eta - A^\ast \iota_X d \zeta
\end{array}
 \right) .
 \end{aligned}
\end{equation}
The first entry is
\[
\begin{aligned}
& A[Z, X] + AX  + \pi^\sharp \mathcal L_Z \eta + \pi^\sharp \iota_X d \zeta \\
& = [Z, AX] - (\mathcal L_Z A)X + AX + [Z, \pi^\sharp \eta] - (\mathcal L_Z \pi)^\sharp \eta + (\pi^\sharp \circ (d \zeta)_\flat) X \\
& = [Z, AX + \pi^\sharp \eta]  + AX + \pi^\sharp \eta \\
& = \mathrm{i} ([Z, X] + X) ,
\end{aligned}
\]
where we used that $AX + \pi^\sharp \eta$ is the first entry of $\mathbbm J \alpha$. Similarly, the second entry in (\ref{eq:comp}) is
\[
\begin{aligned}
& \sigma_\flat [Z, X] + \sigma_\flat X - A^\ast \mathcal L_Z \eta - A^\ast \iota_X d \zeta \\
& = \mathcal L_Z (\sigma_\flat X) - (\mathcal L_Z \sigma)_\flat X + \sigma_\flat X - \mathcal L_Z (A^\ast \eta) + (\mathcal L_Z A)^\ast \eta - (A^\ast \circ (d \zeta)_\flat) X  \\
& = \mathcal L_Z (\sigma_\flat X-A^\ast \eta) + (d \zeta)_\flat (   A X + \pi^\sharp \eta) \\
& = \mathrm{i} (\mathcal L_Z \eta + \iota_X d \zeta),
\end{aligned}
\]
showing that $([Z, X] + X, \mathcal L_Z \eta + \iota_X d \zeta)$ is a $+\mathrm{i}$-eigensection of $\mathbbm J$.

Conversely, let $([Z, X] + X, \mathcal L_Z \eta + \iota_X d \zeta)$ be a $+\mathrm{i}$-eigensection of $\mathbbm J$ for all $
+\mathrm{i}$-eigensections $\alpha = (X, \eta)$. One can show that $(\mathbbm J, \mathbbm Z)$ is a homogeneous generalized 
complex structure with a similar computation as above (but in the reverse order).
\end{proof}

Every homogeneous generalized complex structure $(\mathbbm J, \mathbbm Z)$ determines a complex Dirac-Jacobi structure on the 
trivial line bundle $\mathbbm R_M := M \times \mathbbm R \to M$ according to the following

\begin{proposition}\label{prop:L_J_z}
Let $(\mathbbm J, \mathbbm Z)$ be a homogeneous generalized complex structure on $M$, with $\mathbbm Z = (Z, \zeta)$. In $
\mathbbm D \mathbbm R_M \otimes \mathbbm C$ consider the subbundle $\mathfrak 
L_{(\mathbbm J, \mathbbm Z)}$ spanned over $\mathbbm C$ as follows:
\begin{equation}\label{eq:hgc_DJ}
\mathfrak L_{(\mathbbm J, \mathbbm Z)} := \left\langle \left(1- Z, \zeta + \zeta(Z) \cdot \mathfrak j \right), \left(X, \eta 
+ (\eta (Z) - \zeta (X)) \cdot \mathfrak j \right) : (X, \eta) \in \mathfrak L_{\mathbbm J} \right\rangle
\end{equation}
(where we use the same notations as in Remark \ref{rem:trivial}). Then $\mathfrak L_{(\mathbbm J, \mathbbm Z)}$ is a 
(complex) Dirac-Jacobi structure.
\end{proposition}

\begin{proof}
A direct computation with the generators shows that $\mathfrak L_{(\mathbbm J, \mathbbm Z)}$ is isotropic. As its rank is $\dim M +1$, it is also maximal isotropic. For the involutivity, we will show that the trilinear form
\[
\Upsilon : \wedge^3 \mathfrak L_{(\mathbbm J, \mathbbm Z)} \to \mathbbm R_M, \quad (\alpha, \beta, \gamma) \mapsto \bla \alpha, \blq \beta, \gamma \brq \bra
\]
vanishes on generators. To do this, first denote by $\bla -,- \bra_{\mathbbm T M}$ and $\blq -, - \brq_{\mathbbm T M}$ the bilinear form and the Dorfman bracket in the generalized tangent bundle, and notice that, for all
\[
(X_i + f_i, \eta_i + g_i \cdot \mathfrak j) \in \Gamma (\mathbbm D \mathbbm R_M\otimes \mathbbm C),
\]
with $X_i \in \mathfrak X(M)$, $\eta_i \in \Omega^1 (M)$, and $f_i, g_i \in C^\infty (M)$,  $i = 1, 2$, we have:
\begin{equation}\label{eq:bilin}
\bla (X_1 + f_1, \eta_1 + g_1 \cdot \mathfrak j), (X_2 + f_2, \eta_2 + g_2 \cdot \mathfrak j) \bra = \bla (X_1, \eta_1), (X_2, \eta_2) \bra_{\mathbbm T M} + f_1 g_2 + f_2 g_1,
\end{equation}
and
\begin{equation}\label{eq:Dorfman}
\begin{aligned}
 & \blq (X_1 + f_1, \eta_1 + g_1 \cdot \mathfrak j), (X_2 + f_2, \eta_2 + g_2 \cdot \mathfrak j) \brq =  \blq (X_1, \eta_1), (X_2, \eta_2) \brq_{\mathbbm T M} \\
& \quad + \left(X_1 (f_1) - X_2 (f_1), g_2 df_1 + g_1 df_2 + (X_1 (g_2) - X_2 (g_1) + g_2 f_1 + \eta_1 (X_2))\cdot \mathfrak j \right)
\end{aligned}
\end{equation}
 Now, let $\alpha = (1-Z, \zeta + \zeta (Z) \cdot \mathfrak j)$ and for $(X_i, \eta_i) \in \Gamma (\mathfrak L_{\mathbbm J})$, let 
 \[
 \beta_i = \left(X_i, \eta_i + (\eta_i (Z)- \zeta (X_i))\cdot \mathfrak j\right)
 \]
be the corresponding generator of $\Gamma (\mathfrak L_{(\mathbbm J, \mathbbm Z)})$, $i = 1, 2, 3$. A straightforward computation exploiting (\ref{eq:bilin}) and (\ref{eq:Dorfman}) shows that
 \[
\Upsilon (\beta_1, \beta_2, \alpha) = \bla (X_1, \eta_1), ([Z, X_2], \mathcal L_Z \eta_2 - \eta_2 + \iota_{X_2} d \zeta)\bra_{\mathbbm T M}
 \]
and the right hand side vanishes in view of Proposition \ref{prop:hom_gc}. Finally, again from (\ref{eq:bilin}) and (\ref{eq:Dorfman}) we get
\[
\Upsilon (\beta_1, \beta_2, \beta_3) = \bla (X_1, \eta_1), \blq (X_2, \eta_2), (X_3, \beta_3) \brq \bra_{\mathbbm T M} = 0
\]
and this concludes the proof.
\end{proof}

Complex Dirac-Jacobi structures on $\mathbbm R_M$, of the form $\mathfrak L_{(\mathbbm J, \mathbbm Z)}$ for some homogeneous generalized complex structure $(\mathbbm J, \mathbbm Z)$, can be characterized as follows. First of all, denote by $p_{\mathbbm R} : D\mathbbm R_M = TM \oplus \mathbbm R_M \to \mathbbm R_M$ the natural projection.

\begin{proposition}\label{prop:hom_gen_compl}
A complex Dirac-Jacobi structure $\mathfrak L \subset \mathbbm D \mathbbm R_M\otimes \mathbbm C$ is of the form $\mathfrak L_{(\mathbbm J, \mathbbm Z)}$ for some homogeneous generalized complex structure $(\mathbbm J, \mathbbm Z)$ if and only if it satisfies the following conditions
\begin{enumerate}
\item $\operatorname{rank}_{\mathbbm C} (\mathfrak L \cap \overline{\mathfrak L}) = 1$,
\item $p_D \mathfrak L + p_D \overline{\mathfrak L} = D \mathbbm R_M \otimes \mathbbm C$,
\item $p_{\mathbbm R} \circ p_D : \mathfrak L \cap \overline{\mathfrak L} \to M \times \mathbbm C$ is surjective (hence an isomorphism).
\end{enumerate}
\end{proposition}

\begin{proof}
Begin with a homogeneous generalized complex structure $(\mathbbm J, \mathbbm Z)$, $\mathbbm Z = (Z, \zeta)$, and the associated complex Dirac-Jacobi structure $\mathfrak L = \mathfrak L_{(\mathbbm J, \mathbbm Z)}$ as in (\ref{eq:hgc_DJ}). It is easy to see that $\mathfrak L\cap \overline{\mathfrak L}$ is spanned by $(1-Z, \zeta + \zeta (Z) \cdot \mathfrak j)$, in particular $\mathfrak L$ satisfies property (1) in the statement. For property (2) notice that $p_D \mathfrak L + p_D \overline{\mathfrak L}$ is spanned by $1- Z$ and
\[
p_T \mathfrak L_{\mathbbm J} + p_T \overline{\mathfrak L}_{\mathbbm J} = p_T (\mathfrak L_{\mathbbm J} + \overline{\mathfrak L}_{\mathbbm J})
= TM \otimes \mathbbm C,
\]
where we denoted by $p_T : \mathbbm TM \to TM$ the projection. So $\mathfrak L$ satisfies also (2). Property (3) now follows from the fact that 
\[
p_{\mathbbm R} (1 - Z) = 1 \neq 0.
\]
This concludes the ``only if'' part of the proof.

For the ``if'' part, let $\mathfrak L \subset \mathbbm D \mathbbm R_M\otimes \mathbbm C$ be a complex Dirac-Jacobi structure satisfying properties (1)-(3) in the statement. It follows from (1) and (3) that there exists a unique, necessarily real, section $\alpha$ of $\mathfrak L \cap \overline{\mathfrak L}$ such that $(p_{\mathbbm R} \circ p_D) \alpha = 1$. In particular, $\alpha$ is of the form $(1 - Z, \zeta + g \cdot \mathfrak j)$, for a real vector field $Z$, a real $1$-form $\zeta$, and a real function $g$. From isotropy, $g = \zeta (Z)$, so
\[
\alpha = (1 - Z, \zeta + \zeta (Z) \cdot \mathfrak j), \quad Z \in \mathfrak X(M), \quad \zeta \in \Omega^1 (M).
\]
We put $\mathbbm Z := (Z, \zeta)$. Next we want to construct a generalized complex structure $\mathbbm J : \mathbbm TM \to \mathbbm TM$. To do this, we first define
\[
\mathfrak L_{\mathbbm T} := \left\{(X, \eta) \in \mathbbm T M\otimes \mathbbm C : (X, \eta + (\eta (Z) - \zeta (X)) \cdot \mathfrak j) \in \mathfrak L \right\}.
\]
We claim that $\mathfrak L_{\mathbbm T}$ is a complex Dirac structure such that $\mathfrak L_{\mathbbm T} \cap \overline{\mathfrak L}_{\mathbbm T} = 0$. From (\ref{eq:bilin}), $\mathfrak L_{\mathbbm T}$ is (pointwise) maximal isotropic. So it is a regular vector subbundle provided only it is the image of a vector bundle map. Our next aim is constructing such a map. First of all, consider the endomorphism
\[
F : \mathbbm D \mathbbm R_M \otimes \mathbbm C\to \mathbbm D \mathbbm R_M\otimes \mathbbm C, \quad (X + f, \eta + g \cdot \mathfrak j) \mapsto (X + fZ + f, \eta -f \zeta + g \cdot \mathfrak j) ,
\]
and the natural projection 
\[
p_{\mathbbm T} : \mathbbm D \mathbbm R_M\otimes \mathbbm C \to \mathbbm T M \otimes \mathbbm C, \quad (X + f, \eta + g \cdot \mathfrak j) \mapsto (X, \eta).
\]
We want to show that
\[
\mathfrak L_{\mathbbm T} = (p_{\mathbbm T} \circ F) \mathfrak L.
\]
As $F$ fixes elements of the form $(X, \eta + g \cdot \mathfrak j)$, it is clear that $\mathfrak L_{\mathbbm T} \subset (p_{\mathbbm T} \circ F) \mathfrak L$. In order to check the reverse inclusion, begin with $\beta = (X + f, \eta + g \cdot \mathfrak j) \in \mathfrak L$. It follows from isotropy that $g = \eta(Z) - \zeta (X) - f \zeta (Z)$. Now compute
\[
(\tilde X, \tilde \eta) := (p_{\mathbbm T} \circ F) \beta = (X + fZ, \eta - f \zeta).
\]
But $(\tilde X, \tilde \eta) \in \mathfrak L_{\mathbbm T}$, indeed
\[
\begin{aligned}
& (\tilde X, \tilde \eta + (\tilde \eta (Z) - \zeta(\tilde X)) \cdot \mathfrak j)\\
& = (X + fZ, \eta - fZ + (\eta(Z)  - \zeta (X) - 2f\zeta(Z)) \cdot \mathfrak j) \\
& = \beta - f \alpha
\end{aligned}
\]
which belongs to $\mathfrak L$. So $\mathfrak L_{\mathbbm T}$ is a regular maximal isotropic subbundle of $\mathbbm T M \otimes \mathbbm C$. Involutivity follows from (\ref{eq:Dorfman}) and the involutivity of $\mathfrak L$. Next we check $\mathfrak L_{\mathbbm T} \cap \overline{\mathfrak L}_{\mathbbm T} = 0$. So let $(X, \eta) \in \mathfrak L_{\mathbbm T} \cap \overline{\mathfrak L}_{\mathbbm T}$. This means that
\[
(X, \eta + (\eta(Z) - \zeta (X)) \cdot \mathfrak j) \in \mathfrak L \cap \overline{\mathfrak L}.
\]
As $p_{\mathbbm R} X = 0$ this can only be if $(X, \eta) = 0$. We conclude that $\mathfrak L_{\mathbbm T} $ is the $+\mathrm{i}$-eigenbundle of a generalized complex structure $\mathbbm J$ on $M$. Using (\ref{eq:Dorfman})
again, and Proposition \ref{prop:hom_gc}, it is easy to see that $(\mathbbm J, \mathbbm Z)$ is a homogeneous generalized complex structure in a similar way as in the proof of Proposition \ref{prop:L_J_z}. Finally, it is obvious that $\mathfrak L_{(\mathbbm J, \mathbbm Z)} \subset \mathfrak L$. As they are both maximal isotropic, they actually coincide. This concludes the proof.
\end{proof}

Notice that conditions (1) and (2) in Proposition \ref{prop:hom_gen_compl} make sense for every complex Dirac-Jacobi structure. So we give the following

\begin{definition}
A complex Dirac-Jacobi structure $\mathfrak L \subset \mathbbm D L\otimes \mathbbm C$ on a line bundle $L \to M$ is of \emph{homogeneous generalized complex type} if
\begin{enumerate}
\item $\operatorname{rank}_{\mathbbm C} (\mathfrak L \cap \overline{\mathfrak L}) = 1$,
\item $p_D \mathfrak L + p_D \overline{\mathfrak L} = D L\otimes \mathbbm C$.
\end{enumerate}
\end{definition}

The above definition is motivated by the following

\begin{proposition}
Let  $\mathfrak L \subset \mathbbm D L\otimes \mathbbm C$ be a complex Dirac-Jacobi structure of homogeneous generalized complex type on a line bundle $L \to M$. Then, locally, around every point of $M$, there exists a trivialization $L \cong \mathbbm R_M$ identifying $\mathfrak L$ with the complex Dirac-Jacobi structure $\mathfrak L_{(\mathbbm J, \mathbbm Z)} \subset \mathbbm D \mathbbm R_M\otimes \mathbbm C$ induced by a homogeneous generalized complex structure $(\mathbbm J, \mathbbm Z)$.
\end{proposition}
\begin{proof}\label{prop:loc_hgc_type}
Let $\mathfrak L$ be as in the statement, and let $x_0 \in M$. Choose a nowhere vanishing local section $\alpha = (\Delta, \psi) \in \Gamma (\mathfrak L \cap \overline{\mathfrak L})$ around $x_0$. We can choose $(\Delta, \psi)$ to be real. Then we have $\Delta \neq 0$. Indeed, if $\Delta_{x} = 0$ for some $x$, then 
\[
0 \neq \psi_{x} \in \mathfrak L \cap J^1 L \subset \mathsf{Ann} (p_D \mathfrak L) \cap \mathsf{Ann} (p_D \overline{\mathfrak L}) = \mathsf{Ann} (p_D \mathfrak L + p_D \overline{\mathfrak L}) = \mathsf{Ann} (D L \otimes \mathbbm C) = 0,
\]
a contradiction. So $\Delta $ is a non-vanishing (local) derivation. It is easy to see that, for a non-vanishing derivation of $L$, locally, around every point, there always exists a trivialization $L \cong \mathbbm R_M$ identifying $\Delta$ with a derivation of the form $f (1 - Z)$ with $f$ a nowhere vanishing function. As $p_{\mathbbm R} (f(1 -Z)) = f \neq 0$, this is the trivialization we where looking for.
\end{proof}

\begin{remark}
Let $\mathfrak L$, $x_0$ and $\Delta$ be as in the proof of Proposition \ref{prop:loc_hgc_type}. If $\Delta$ can be chosen so that $\Delta_{x_0} = \mathbb 1_{x_0}$, then \emph{every} local trivialization $L \cong \mathbbm R_M$ around $x_0$ identifies $\mathfrak L$ with the complex Dirac-Jacobi structure induced by a homogeneous generalized complex structure.
\end{remark}

\section{The transversal to a leaf}\label{Sec:transv}

Let $(L \to M, \mathbbm K)$ be a generalized contact bundle with 
\[
\mathbbm K = 
\left(
\begin{array}{cc}
\varphi & J^\sharp \\
\omega_\flat & - \varphi^\dag
\end{array}
\right),
\]
 and let $\mathfrak L$ be its $+\mathrm{i}$-eigenbundle. In this sections, as a preparation for the splitting theorems, we study special classes of submanifolds of $M$. Specifically, characteristic leaves of the underlying Jacobi structure $J$ and their \emph{transversals}. As we already outlined, in this paper, by a \emph{transversal} to a leaf $\mathcal O$ at a point $x_0 \in \mathcal O$, we will always understand a \emph{minimal dimension} transversal, i.e.~a submanifold $N$ through $x_0$ such that $T_{x_0} M = T_{x_0} N \oplus T_{x_0} \mathcal O$. We begin with \emph{contact leaves}.

\subsection{Contact leaves and their transversals}\label{Sec:Transv_contact}

Recall that an odd dimensional characteristic leaf $\mathcal O$ of $J$ possesses a canonical contact structure $H \subset T \mathcal O$. This can be seen as follows. First of all, 
$J$ restricts to a Jacobi structure $J_{\mathcal O}$ on the restricted line bundle $L_{\mathcal O} := L|_{\mathcal O} \to \mathcal O$. Now $ \sigma J_{\mathcal O}^\sharp : J^1 L_{\mathcal O} \to T\mathcal O$ is surjective (by definition of characteristic leaf), and it follows from $\dim \mathcal O = \mathrm{odd}$ that $J_{\mathcal O}^\sharp : J^1 L_{\mathcal O} \to D L_{\mathcal O}$ is surjective, hence an isomorphism. Let $\omega_{\mathcal O} = J_{\mathcal O}^{-1} \in \Omega^2_{L_{\mathcal O}}$ be the Atiyah $2$-form inverting $J_{\mathcal O}$, i.e.~$(\omega_{\mathcal O})_\flat := (J_{\mathcal O}^\sharp)^{-1}$. Notice that $DL_{\mathcal O} = (\operatorname{im} J^\sharp) |_{\mathcal O} $ and $\omega_{\mathcal O}$ agrees with the pointwise restriction to $\mathcal O$ of the $2$-form $\omega_J : \wedge^2 \operatorname{im} J^\sharp \to L$ from Remark \ref{rem:impi}. Now, the integrability condition for $J_{\mathcal O}$ is equivalent to $d_D \omega_{\mathcal O} = 0$, and $\iota_{\mathbb 1} \omega_{\mathcal O} \in \Omega^1_{L_{\mathcal O}}$ is necessarily of the form $\iota_{\mathbb 1} \omega_{\mathcal O} = \theta_{\mathcal O} \circ \sigma$ for a unique $L_{\mathcal O}$-valued $1$-form $\theta_{\mathcal O} : T\mathcal O \to L_{\mathcal O}$. The kernel $H$ of $\theta_{\mathcal O}$ is a contact structure containing the full information on $J_{\mathcal O}$. This contact structure can be equivalently encoded in a generalized contact structure 
\[
\mathbbm K_{\mathcal O} = 
\left(
\begin{array}{cc}
0& J_{\mathcal O}^\sharp \\
- (J^{-1}_{\mathcal O})_\flat & 0
\end{array}
\right)
\]
on $L_{\mathcal O}$. Let $\mathfrak L_{\mathcal O} \subset \mathbbm D L_{\mathcal O}\otimes \mathbbm C$ be the $+\mathrm{i}$-eigenbundle of $\mathbbm K_{\mathcal O}$. We have
\begin{equation}\label{eq:L_O}
\mathfrak L_{\mathcal O} := \left\{ (J_{\mathcal O}^\sharp (\psi), \mathrm{i} \psi) \in \mathbbm D L_{\mathcal O}\otimes \mathbbm C : \psi \in J^1 L_{\mathcal O} \otimes \mathbbm C \right\}.
\end{equation}

In the following, for an (immersed) submanifold $S \hookrightarrow M$, we simply denote by $L_S$ the restricted line bundle $L|_S$, and by $I_{S} : L|_{S} \hookrightarrow L$ the natural (injective) immersion. It is a regular vector bundle map covering the injective immersion $i_{S} : S \hookrightarrow M$.

\begin{proposition}\label{prop:contact_leaf}
Let $\mathcal O$ be an odd dimensional leaf of $J$. The backward image of the complex Dirac-Jacobi structure $\mathfrak L$ along the immersion $I_{\mathcal O} : L_{\mathcal O} \hookrightarrow L$ is a Dirac-Jacobi structure of generalized contact type on $L_{\mathcal O}$, denoted $I^!_{\mathcal O} \mathfrak L$. Additionally, it is a $B$-field transformation of $\mathfrak L_{\mathcal O}$:
\[
I^!_{\mathcal O} \mathfrak L = \mathfrak L_{\mathcal O}^B
\]
for some $B \in Z^2_{L_{\mathcal O}}$.
\end{proposition}

\begin{proof}
We divide the proof in several steps. First we prove that $I^!_{\mathcal O} \mathfrak L \subset \mathbbm DL_{\mathcal O}\otimes \mathbbm C$ is a regular subbundle, checking the clean intersection condition (\ref{eq:cic}):
\[
\operatorname{rank}_{\mathbbm C} \left(p_D \mathfrak L|_{\mathcal O} + DL_{\mathcal O} \otimes \mathbbm C\right) = \text{constant}.
\]
To do this notice that 
\[
D L_{\mathcal O} \otimes \mathbbm C = \operatorname{im} J^\sharp|_{\mathcal O} \otimes \mathbbm C = p_D \mathfrak L|_{\mathcal O} \cap p_D\overline{ \mathfrak L}|_{\mathcal O} \subset p_D \mathfrak L|_{\mathcal O}.
\]
Hence 
\[
p_D \mathfrak L|_{\mathcal O} + DL_{\mathcal O} \otimes \mathbbm C = p_D \mathfrak L|_{\mathcal O}
\]
which is constant rank because
\begin{enumerate}
\item $p_D \mathfrak L|_{\mathcal O} + p_D\overline{ \mathfrak L}|_{\mathcal O} = (DL)|_{\mathcal O} \otimes \mathbbm C$ is constant rank,
\item $p_D \mathfrak L|_{\mathcal O} \cap p_D\overline{ \mathfrak L}|_{\mathcal O} = D L_{\mathcal O} \otimes \mathbbm C$ is constant rank, and
\item $p_D \mathfrak L|_{\mathcal O}$ and $p_D\overline{ \mathfrak L}|_{\mathcal O}$ have the same rank.
\end{enumerate}
This proves the first part of the statement.

Next we show that $I^!_{\mathcal O} \mathfrak L$ is a Dirac-Jacobi structure of generalized contact type. For this, it is enough to check that 
\[
I^!_{\mathcal O} \mathfrak L \cap \overline{I^!_{\mathcal O} \mathfrak L} = 
I^!_{\mathcal O} \mathfrak L \cap I^!_{\mathcal O} \overline{\mathfrak L} = 0.
\]
So let $(\Delta, \psi) \in (I^!_{\mathcal O} \mathfrak L \cap I^!_{\mathcal O} \overline{\mathfrak L})_x $ for some $x \in \mathcal O$. This means that, there exist $\chi', \chi'' \in J^1_x L \otimes \mathbbm C$ such that $\psi = I^\ast_{\mathcal O} \chi' =  I^\ast_{\mathcal O} \chi''$, and, additionally, $(\Delta, \chi') \in \mathfrak L_x $, and $(\Delta, \chi'') \in \overline{\mathfrak L}_x$. Now, define $\chi \in J^1_x L \otimes \mathbbm C$ by putting
\[
\langle \chi, \nabla \rangle = \left\{
\begin{array}{cl}
\langle \chi' , \nabla \rangle & \text{if $\nabla \in p_D \mathfrak L_x$} \\
\langle \chi'' , \nabla \rangle & \text{if $\nabla \in p_D \overline{\mathfrak L}_x$}
\end{array}
 \right. .
\]
As both $\chi'$ and $\chi''$ agree with $\psi$ on $(p_D \mathfrak L \cap p_D \overline{\mathfrak L})_x = (\operatorname{im} J^\sharp)_x \otimes \mathbbm C = D_x L_{\mathcal O}\otimes \mathbbm C$, then $\chi$ is well-defined. It immediately follows from (\ref{eq:L_pi}) that $(\Delta, \chi) \in (\mathfrak L \cap \overline{\mathfrak L})_x = 0$. So $(\Delta, \chi) = 0$, hence $(\Delta, \psi) = 0$. We conclude that $I^!_{\mathcal O} \mathfrak L$ is a Dirac-Jacobi structure of generalized contact type. In particular, there is an underlying Jacobi structure $\tilde J$ on $L_{\mathcal O}$.

As a third step, we prove that the Jacobi structure underlying $I^!_{\mathcal O} \mathfrak L$ is precisely $J_{\mathcal O}$: the restriction to $\mathcal O$ of the Jacobi structure $J$. In other words, $\tilde J = J_{\mathcal O}$. First of all,
\[
p_{D} I^!_{\mathcal O} \mathfrak L = p_D \mathfrak L \cap (D L_{\mathcal O} \otimes \mathbbm C) = D L_{\mathcal O} \otimes \mathbbm C.
\]
In particular the $L_{\mathcal O}$-valued $2$-form $\varpi_{\mathcal O}$ induced by $I^!_{\mathcal O} \mathfrak L$ on $p_{D} I^!_{\mathcal O} \mathfrak L$ is a genuine (complex) Atiyah $2$-form on $L_{\mathcal O}$. Notice that $\varpi_{\mathcal O}$ actually agrees with $\varpi$ on $D L_{\mathcal O}$. Indeed let $\Delta, \nabla \in DL_{\mathcal O}$. There is $\psi \in J^1 L$ such that $(\Delta, I_{\mathcal O}^\ast \psi) \in I^!_{\mathcal O} \mathfrak L$. Compute
\[
\varpi_{\mathcal O} (\Delta, \nabla) = \langle  I_{\mathcal O}^\ast \psi, \nabla \rangle
= \langle \psi, \nabla \rangle = \varpi (\Delta, \nabla).
\]

Now, let $\psi \in J^1 L_{\mathcal O}$, and let $\Psi \in J^1 L$ be such that $I^\ast_{\mathcal O} \Psi = \psi$. We want to compare $J_{\mathcal O}^\sharp \psi$ and $\tilde J{}^\sharp \psi$. To do this pick $\nabla \in D L_{\mathcal O}$ and compute
\begin{gather*}
\operatorname{Im} \varpi_{\mathcal O} (\tilde J{}^\sharp \psi , \nabla) =  \langle \psi, \nabla \rangle = \langle  I_{\mathcal O}^\ast \Psi, \nabla \rangle = \langle \Psi, \nabla \rangle \\ = \operatorname{Im} \varpi (J^\sharp \Psi, \nabla)
= \operatorname{Im} \varpi (J^\sharp_{\mathcal O} \psi, \nabla) = \operatorname{Im} \varpi_{\mathcal O} (J^\sharp_{\mathcal O} \psi, \nabla).
\end{gather*}
But $\operatorname{Im} \varpi_{\mathcal O}$ is non-degenerate (Remark \ref{rem:impi}) so $\tilde J^\sharp \psi = J^\sharp_{\mathcal O} \psi$. In particular $\operatorname{Im} \varpi_{\mathcal O} = J^{-1}_{\mathcal O}$.

Finally, we prove that $I^!_{\mathcal O} \mathfrak L$ is a $B$-field transformation of $\mathfrak L_{\mathcal O}$. From $p_D I^!_{\mathcal O} \mathfrak L = D L_{\mathcal O} \otimes \mathbbm C$, we have
\begin{equation}\label{eq:IL}
I^!_{\mathcal O} \mathfrak L = \left\{(\Delta, \iota_\Delta \varpi_{\mathcal O}) : \Delta \in DL_{\mathcal O} \right\} = \operatorname{\mathsf{graph}} (\varpi_{\mathcal O})_\flat \subset \mathbbm D L_{\mathcal O}\otimes \mathbbm C.
\end{equation}
Let $B = \operatorname{Re} \varpi_{\mathcal O} \in \Omega^2_{L_{\mathcal O}}$.
From (\ref{eq:IL}), and the involutivity of $I^{!}_{\mathcal O} \mathfrak L$ we have $d_D \varpi_{\mathcal O} = 0$, hence $d_D B = 0$. Finally compute
\[
(I^!_{\mathcal O} \mathfrak L)^{- B} = \operatorname{\mathsf{graph}} (\varpi_{\mathcal O} - \operatorname{Re} \varpi_{\mathcal O}) = \operatorname{\mathsf{graph}} (\mathrm{i} \operatorname{Im} \varpi_{\mathcal O}) = \mathfrak L_{\mathcal O},
\]
where we used (\ref{eq:L_O}) and the fact that $\operatorname{Im} \varpi_{\mathcal O} = J^{-1}_{\mathcal O}$.
\end{proof}

We now pass to transversals. A transversal to a characteristic leaf of a generalized complex manifold inherits a generalized complex structure, at least around the intersection point with the leaf. The precise analogue cannot be true for contact leaves of a generalized contact structure, simply because, in this case, transversals are even dimensional. 

\begin{proposition}\label{prop:contact_transversal}
Let $N$ be a transversal at $x_0 \in \mathcal O$ to an odd dimensional leaf $\mathcal O$ of $J$. Around $x_0$, the backward image of the complex Dirac-Jacobi structure $\mathfrak L$ along the embedding $I_N : L_N \hookrightarrow L$ is a complex Dirac-Jacobi structure $I^!_N \mathfrak L$ of homogeneous generalized complex type such that 
\[
\left(I^!_N \mathfrak L \cap \overline{I^!_N \mathfrak L}\right)_{x_0}
\]
 is spanned by a vector of the form $(\mathbb 1_{x_0}, \psi)$. In particular, any local trivialization $L_N \cong \mathbbm R_N$ around $x_0$, identifies $I^!_N \mathfrak L$  with the complex Dirac-Jacobi structure corresponding to a homogeneous generalized complex structure.
\end{proposition}

\begin{proof}
First of all we prove that $I^{!}_N \mathfrak L \subset \mathbbm D L_N$ is a regular subbundle, hence a Dirac-Jacobi structure on $L_N$, checking the clean intersection condition (\ref{eq:cic}):
\[
\operatorname{rank}_{\mathbbm C}(p_D \mathfrak L|_N + DL_N \otimes \mathbbm C) = \mathrm{constant}.
\]
We have
\[
p_D \mathfrak L|_N \supset p_D \mathfrak L|_N \cap p_D \overline{\mathfrak L}|_N = \operatorname{im} J^\sharp |_N \otimes \mathbbm C.
\]
At the point $x_0$ we have $(\operatorname{im} J^\sharp)_{x_0} \otimes \mathbbm C = D_{x_0} L_{\mathcal O}$, so
\[
p_D \mathfrak L_{x_0} + D_{x_0}L_N \otimes \mathbbm C \supset D_{x_0} L_{\mathcal O} + D_{x_0} L_N = D_{x_0} L.
\]
But $p_D \mathfrak L|_N + DL_N \otimes \mathbbm C \subset (DL)|_N$ is a smooth, possibly non-regular subbundle, hence its rank can only increase around $x_0$, and we conclude that
\begin{equation}\label{eq:sum}
p_D \mathfrak L|_N + DL_N \otimes \mathbbm C = (DL)|_N \otimes \mathbbm C
\end{equation}
in a whole neighborhood of $x_0$. In particular, the left hand side has constant rank.

Next we show that $\operatorname{rank}_{\mathbbm C}(I^!_N \mathfrak L \cap I^!_N \overline{\mathfrak L}) = 1$ around $x_0$. Denote by $\nu N = TM|_N / TN$ and $\nu^\ast N = \mathsf{Ann}(TN)\subset T^\ast M|_N$ the normal and the conormal bundle to $N$, respectively. It is useful to consider the following skew-symmetric bilinear map
\[
\mu : \wedge^2 \left(\nu^\ast N \otimes L_N\right) \to L_N, \quad (\eta, \theta) \mapsto \langle J^\sharp\eta,  \theta \rangle,
\]
and the associated vector bundle map $\mu^\sharp : \nu^\ast N \otimes L_N \to \nu N$ implicitly defined by
\[
\langle \theta, \mu^\sharp \eta  \rangle = \mu (\eta, \theta), \quad \eta, \theta \in \nu^\ast N \otimes L_N.
\]
In other words $\mu^\sharp$ is the composition 
\[
\nu^\ast N \otimes L_N \overset{J^\sharp}{\longrightarrow} (D L)|_N \overset{\sigma}{\longrightarrow} TM|_N \longrightarrow \nu N,
\]
where the last arrow is the natural projection (with kernel $TN$). We want to show that $\mu$ has maximal rank around $x_0$: that is $\operatorname{rank} \mu = \operatorname{rank}(\nu N) - 1 = \dim \mathcal O -1 = \mathrm{even}$. To do this it is enough to show that $\operatorname{rank}_{x_0} \mu =  \dim \mathcal O -1$, in other words $\dim (\ker \mu^\sharp_{x_0}) = 1$. So compute
\[
\ker \mu^\sharp_{x_0} = \left\{\eta \in \nu^\ast_{x_0} N \otimes L_{x_0} : J_{x_0}^\sharp \eta \in DL_N \right\}.
\]
But $\nu^\ast_{x_0} N \otimes L_{x_0} = T^\ast_{x_0} \mathcal O \otimes L_{x_0}$, and $(\operatorname{im} J^\sharp)_{x_0} = D_{x_0} L_{\mathcal O}$, so we find
\begin{equation}\label{eq:x_0}
\begin{aligned}
\ker \mu^\sharp_{x_0} & = \left\{\eta \in T_{x_0} \mathcal O \otimes L_{x_0} : J_{x_0}^\sharp \eta = r \cdot \mathbb 1_{x_0}, \text{for some $r \in \mathbbm R$}\right\}\\
& = (J_{\mathcal O}^{-1})_{\flat} \langle \mathbb 1_{x_0} \rangle = \langle (\theta_{\mathcal O})_{x_0} \rangle.
\end{aligned}
\end{equation}
Now, we go back to $I^!_N \mathfrak L$ and consider the real (a-priori not necessarily regular) subbundle
\[
R := \left\{\operatorname{Re} \alpha : \alpha \in I^!_N \mathfrak L \cap I^!_N \overline{\mathfrak L}  \right\} \subset I^!_N \mathfrak L \cap I^!_N \overline{\mathfrak L}.
\]
clearly $I^!_N \mathfrak L \cap I^!_N \overline{\mathfrak L}$ is (canonically isomorphic to) the complexification of $R$. We want to show that there is a (pointwise) exact sequence:
\begin{equation}\label{eq:es}
0 \longrightarrow R \overset{\kappa}{\longrightarrow} \nu^\ast N \otimes L_N \overset{\mu^\sharp}{\longrightarrow} \nu N,
\end{equation}
proving that, around $x_0$, $\operatorname{rank}_{\mathbbm R} R = \operatorname{rank}_{\mathbbm C}(I^!_N \mathfrak L \cap I^!_N \overline{\mathfrak L}) = 1$ as claimed. Define $\kappa : R \to \nu^\ast N \otimes L_N$ as follows. Take $\alpha = (\Delta, \psi) \in R$. This means that $(\Delta, \psi) \in \mathbbm D L_N$ is such that there exists $\chi \in J^1 L \otimes \mathbbm C$ with $(\Delta, \chi) \in \mathfrak L|_N$ (hence $(\Delta, \overline \chi) \in \overline{\mathfrak L}|_N$) and $I^\ast_N \chi = \psi$. Actually, $\chi$ is unique. Indeed, let $\chi' \in J^1 L \otimes \mathbbm C$ be such that $(\Delta, \chi') \in \mathfrak L|_N$ and $I^\ast_N \chi' = \psi$. Then, on one side
\[
\chi - \chi' \in \left((J^1 L)|_N \otimes \mathbbm C \right)\cap \mathfrak L|_N = \mathsf{Ann} \left(p_D \mathfrak L|_N\right).
\]
On the other side, $I_N^\ast (\chi - \chi') = 0$, i.e.
\[
\chi - \chi' \in \mathsf{Ann} \left( D L_N \otimes \mathbbm C\right),
\]
so
\[
\chi - \chi'  \in \mathsf{Ann} \left( p_D \mathfrak L|_N + D L_N \otimes \mathbbm C \right) = 0
\]
where we used (\ref{eq:sum}) (which holds true around $x_0$). So if we work around $x_0$, $\chi = \chi'$, we put
\[
\kappa (\Delta, \psi) := \operatorname{Im} \chi,
\]
and, from $\psi = I^\ast_N \chi = I^\ast_N \overline \chi$, it belongs to $\nu^\ast N \otimes L_N = \mathsf{Ann} (D L_N) \subset (J^1 L)|_N$. 

Before proving that the sequence (\ref{eq:es}) is exact, the following remark is useful.  Let $(\Delta,\psi) \in R$ and let $\chi$ be as above. Then 
\begin{equation}\label{eq:Delta,J,psi}
\Delta = J^\sharp (\operatorname{Im} \chi).
\end{equation}
Indeed, from $\mathbbm K (\Delta, \chi) = \mathrm{i} (\Delta, \chi)$, we find $\mathrm{i} \Delta = \varphi \Delta + J^\sharp \chi$ (take just the first component). Similarly, from $\mathbbm K (\Delta, \overline \chi) = -\mathrm{i} (\Delta, \overline \chi)$, we find $\mathrm{i} \Delta = - \varphi \Delta - J^\sharp \overline \chi$. So $\Delta = J^\sharp (\chi - \overline \chi)/2\mathrm{i} = J^\sharp(\operatorname{Im} \chi)$.

Now, we prove that (\ref{eq:es}) is exact. First of all $\kappa$ is injective. Indeed, if $\kappa (\Delta, \psi) = \operatorname{Im} \chi = 0$, then $\chi = \overline \chi$ and $(\Delta, \chi) \in \mathfrak L \cap \overline{\mathfrak L} = 0$, so $\chi = 0$, and, from (\ref{eq:Delta,J,psi}), $(\Delta,\psi) = (0, I_N^\ast \chi) = 0$. It remains to show that $\ker \mu^\sharp = \operatorname{im} \kappa$. So let $(\Delta, \psi) \in R$, let $\chi$ be as above, and let $\eta \in \nu^\ast N \otimes L_N$. Compute
\[
\langle \mu^\sharp (\operatorname{Im} \chi) , \eta \rangle = \langle J^\sharp  (\operatorname{Im} \chi), \eta \rangle = \langle \Delta, \eta \rangle = 0,
\]
where we used (\ref{eq:Delta,J,psi}) again, and the fact that $\Delta \in DL_N$. So $\ker \mu^\sharp \subset \operatorname{im} \kappa$. Finally, let $\eta \in \nu^\ast N \otimes L_N$ be such that $\mu^\sharp \eta = 0$. This means that $\Delta := J^\sharp \eta \in D L_N$. Put
\[
\alpha := (\Delta, -I^\ast_N (\varphi^\dag \eta)).
\]
We claim that $\alpha \in R$, and $\eta = \kappa (\alpha)$. To see this notice that
\[
(\Delta, \mathrm{i} \eta - \varphi^\dag \eta) = \mathrm{i} \left(\mathrm{id} - \mathrm{i} \mathbbm K \right)(0, \eta) \in \mathfrak L
\]
hence $(\Delta, I^\ast_N (\mathrm{i} \eta - \varphi^\dag \eta)) = \alpha \in R$. Additionally
\[
\kappa (\alpha) = \operatorname{Im} (\mathrm{i} \eta - \varphi^\dag \eta) = \eta.
\]
We conclude that $\ker \mu^\sharp = \operatorname{im} \kappa$, and $\operatorname{rank}_{\mathbbm C} (I^!_N \mathfrak L \cap I^!_N \overline{\mathfrak L}) = 1$ as claimed.

To prove that $I^!_N \mathfrak L$ is a complex Dirac-Jacobi structure of homogeneous generalized complex type, it remains to show that $p_D I^!_N \mathfrak L + p_D I^!_N \overline{\mathfrak L} = D L_N$. To do this we compute
\[
\mathsf{Ann} \left( p_D I^!_N \mathfrak L + p_D I^!_N \overline{\mathfrak L} \right)  =
\mathsf{Ann} \left( p_D I^!_N \mathfrak L \right)\cap \mathsf{Ann} \left(p_D I^!_N \overline{\mathfrak L} \right) 
 = \left(J^1 L_N \otimes \mathbbm C\right) \cap I^!_N \mathfrak L \cap I^!_N \overline{\mathfrak L}.
\]
But the above discussion, together with formula (\ref{eq:x_0}), reveals that, at the point $x_0$, $R$, hence $I^!_N \mathfrak L \cap I^!_N \overline{\mathfrak L}$, is spanned by an element of the form $(\mathbb 1_{x_0}, \zeta)$. In particular, 
\[
\left(J^1 L_N \otimes \mathbbm C\right) \cap I^!_N \mathfrak L \cap I^!_N \overline{\mathfrak L} = 0
\]
at the point $x_0$, hence in a whole neighborhood of $x_0$. This concludes the proof.
\end{proof}

\subsection{Lcs leaves and their transversals}\label{Sec:Transv_lcs}

We now pass to lcs leaves. As already mentioned, an even dimensional characteristic leaf of $J$ possesses a canonical lcs structure. To see this one can argue as follows. As before, $J$ restricts to a Jacobi structure $J_{\mathcal O}$ on the restricted line bundle $L_{\mathcal O} \to \mathcal O$. Now $ \sigma J_{\mathcal O}^\sharp : J^1 L_{\mathcal O} \to T\mathcal O$ is surjective again, and, as $\dim \mathcal O = \mathrm{even}$, then $J_{\mathcal O}^\sharp : J^1 L_{\mathcal O} \to D L_{\mathcal O}$ takes values in a $\dim \mathcal O$-dimensional subbundle $C_{\mathcal O} \subset D L_{\mathcal O}$ transversal to $\mathbbm R_{\mathcal O} \subset D L_{\mathcal O}$. In other words, $C_{\mathcal O}$ is the image of a linear connection $\nabla : T \mathcal O \to DL_{\mathcal O}$. Additionally, $J_{\mathcal O}$ induce a non-degenerate $L_{\mathcal O}$-valued $2$-form on $C_{\mathcal O}$, hence on $T \mathcal O$. So we get a non-degenerate $\Omega_{\mathcal O} \in \Omega^2 (\mathcal O, L_{\mathcal O})$. Notice that $C_{\mathcal O} = \operatorname{im} J^\sharp{}|_{\mathcal O} $, and $\Omega_{\mathcal O}$ (viewed as a $2$-form on $C_{\mathcal O}$) agrees with the pointwise restriction to $\mathcal O$ of the $2$-form $\omega_J : \wedge^2 \operatorname{im} J^\sharp \to L$ from Remark \ref{rem:impi}. Now the integrability condition for $J_{\mathcal O}$ is equivalent to $\nabla$ being flat and $d^\nabla \Omega_{\mathcal O} = 0$. So $(\Omega_\mathcal O, \nabla)$ is an lcs structure on $L_{\mathcal O}$ and it contains the full information on $J_{\mathcal O}$.
This lcs structure $(\Omega_{\mathcal O}, \nabla)$ can be equivalently encoded in a complex Dirac-Jacobi structure $\mathfrak L_{\mathcal O}$ given by the same formula (\ref{eq:L_O}) as before.

\begin{remark}
Let $\mathcal O$ be an even dimensional leaf of $J$. Then the subbundle $\mathfrak L_{\mathcal O} \subset \mathbbm D L_{\mathcal O}\otimes \mathbbm C$ given by formula (\ref{eq:L_O}) is a complex Dirac-Jacobi structure such that 
\[
p_D \mathfrak L_{\mathcal O} = p_D \overline{\mathfrak L}_{\mathcal O}  = C_{\mathcal O} \otimes \mathbbm C  \quad \text{and} \quad \mathfrak L_{\mathcal O} \cap \overline{\mathfrak L}_{\mathcal O}  = \mathsf{Ann} (C_{\mathcal O} \otimes \mathbbm C).
\]
\end{remark}

\begin{proposition}
Let $\mathcal O$ be an even dimensional leaf of $J$. The backward image of the complex Dirac-Jacobi structure $\mathfrak L$ along the immersion $I_{\mathcal O} : L_{\mathcal O} \hookrightarrow L$ is a Dirac-Jacobi structure on $L_{\mathcal O}$, denoted $I^!_{\mathcal O} \mathfrak L$, such that 
\begin{enumerate}
\item $\operatorname{rank}_{\mathbbm C} (I^!_{\mathcal O}\mathfrak L \cap I^!_{\mathcal O}\overline{\mathfrak L}) = 1$,
\item $\mathbb 1 \notin p_D I^!_{\mathcal O} \mathfrak L + p_D I^!_{\mathcal O}\overline{\mathfrak L}$, and 
\item $(p_D  I^!_{\mathcal O}\mathfrak L + p_D I^!_{\mathcal O}\overline{\mathfrak L}) \oplus \langle \mathbb 1 \rangle = DL_{\mathcal O} \otimes \mathbbm C$.
\end{enumerate}
Additionally, it is locally a $B$-field transformation of $\mathfrak L_{\mathcal O}$:
\[
I^!_{\mathcal O} \mathfrak L = \mathfrak L_{\mathcal O}^B
\]
for some $B \in Z^2_{L_{\mathcal O}}$.
\end{proposition}

\begin{proof}
First of all, we prove that $I^!_{\mathcal O} \mathfrak L \subset \mathbbm D L_{\mathcal O}\otimes \mathbbm C$ is a regular subbundle. As usual, we check the clean intersection condition:
\[
\operatorname{rank}_{\mathbbm C} \left(p_D \mathfrak L|_{\mathcal O} + DL_{\mathcal O} \otimes \mathbbm C \right) = \mathrm{constant}.
\]
So notice that, in this case
\[
DL_{\mathcal O} \otimes \mathbbm C = \langle \mathbb 1 \rangle \oplus C_{\mathcal O} =\langle \mathbb 1 \rangle \oplus \operatorname{im} J^\sharp|_{\mathcal O} \otimes \mathbbm C \subset \langle \mathbb 1 \rangle + p_D \mathfrak L |_{\mathcal O}.
\]
But $\mathbb 1 \notin p_D \mathfrak L|_{\mathcal O}$, otherwise, from 
$\mathbb 1 = \overline{\mathbb 1}$, and $\operatorname{im} J^\sharp \otimes \mathbbm C = p_D \mathfrak L \cap p_D \overline{\mathfrak L}$ 
we would get $\mathbb 1 \in C_{\mathcal O}$ 
which is not the case. We conclude that
\[
p_D \mathfrak L|_{\mathcal O} + D L_{\mathcal O} \otimes \mathbbm C = \langle \mathbb 1 \rangle \oplus p_D \mathfrak L|_{\mathcal O}
\]
which is constant rank in the same way as for contact leaves (see the proof of Proposition \ref{prop:contact_leaf}). So $I^!_{\mathcal O} \mathfrak L$ is a Dirac-Jacobi structure on $L_{\mathcal O}$.

Next we show that
\begin{equation}\label{eq:equation_1}
p_D I^!_{\mathcal O} \mathfrak L = p_D I^!_{\mathcal O} \overline{\mathfrak L} = C_{\mathcal O} \otimes \mathbbm C.
\end{equation}
We will get, in particular, properties (2) and (3) in the statement.
From $p_D I^!_{\mathcal O} \mathfrak L = D L_{\mathcal O} \cap p_D \mathfrak L$, and $p_D \mathfrak L \cap p_D \overline{\mathfrak L} = \operatorname{im} J^\sharp \otimes \mathbbm C$ we get $C_{\mathcal O} \otimes \mathbbm C \subset p_D I^!_{\mathcal O} \mathfrak L \cap p_D I^!_{\mathcal O} \overline{\mathfrak L}$. Now let $(\Delta, \psi) \in I^!_{\mathcal O} \mathfrak L$, so that $\Delta \in p_D I^!_{\mathcal O}\mathfrak L$. In particular, $\Delta \in D L_{\mathcal O} \otimes \mathbbm C$, meaning that $\Delta = \Delta_0 + z \cdot \mathbb 1$ for some $\Delta_0 \in C_{\mathcal O}\otimes \mathbbm C$, and some $z \in \mathbbm C$. It follows that $\Delta - \Delta_0 \in p_D \mathfrak L|_{\mathcal O}$. As $\mathbb 1 \notin p_D \mathfrak L|_{\mathcal O}$, we have $z = 0$, and $\Delta = \Delta_0 \in C_{\mathcal O}\otimes \mathbbm C$. So $p_D I^!_{\mathcal O} \mathfrak L \subset C_{\mathcal O}\otimes \mathbbm C$, and, similarly, $p_D I^!_{\mathcal O} \overline{\mathfrak L} \subset C_{\mathcal O}\otimes \mathbbm C$.

Now we show that 
\[
I^!_{\mathcal O} \mathfrak L \cap I^!_{\mathcal O} \overline{\mathfrak L} = \mathsf{Ann} (C_{\mathcal O} \otimes \mathbbm C).
\]
We will get, in particular, property (1) in the statement.
So let $(\Delta, \psi) \in (I^!_{\mathcal O} \mathfrak L \cap I^!_{\mathcal O} \overline{\mathfrak L} )_x$ for some $x \in \mathcal O$. This means that there exist $\chi', \chi''$ as in the proof of Proposition (\ref{prop:contact_leaf}), and we can even construct $\chi$ exactly as there. As $\mathfrak L$ is of generalized contact type, actually $(\Delta, \chi) = 0$, i.e.~$\Delta = 0$, and $\psi \in \mathsf{Ann} (C_{\mathcal O} \otimes \mathbbm C)$. This shows that $I^!_{\mathcal O} \mathfrak L \cap I^!_{\mathcal O} \overline{\mathfrak L} \subset \mathsf{Ann} (C_{\mathcal O} \otimes \mathbbm C) \subset J^1 L_{\mathcal O} \otimes \mathbbm C$. The reverse inclusion $\mathsf{Ann} (C_{\mathcal O} \otimes \mathbbm C) \subset I^!_{\mathcal O} \mathfrak L \cap I^!_{\mathcal O} \overline{\mathfrak L}$ immediately follows from (\ref{eq:equation_1}). 

It remains to show that, locally, $I^!_{\mathcal O} \mathfrak L$ is a $B$-field transformation of 
\begin{equation}\label{eq:L_02}
\mathfrak L_{\mathcal O} = \left\{ (J^\sharp_{\mathcal O} (\psi), \mathrm{i} \psi) \in \mathbbm D L_{\mathcal O} \otimes \mathbbm C: \psi \in J^1 L_{\mathcal O} \otimes \mathbbm C \right\}.
\end{equation}
To do this, denote by $\varpi_{\mathcal O}$ the $L_{\mathcal O}$-valued $2$-form induced by $I^!_{\mathcal O} \mathfrak L$ on $p_D I^!_{\mathcal O} \mathfrak L = C_{\mathcal O} \otimes \mathbbm C$. We can extend $\varpi_{\mathcal O}$ to a genuine Atiyah $2$-form on $L_{\mathcal O}$, by putting $\iota_{\mathbb 1} \varpi_{\mathcal O} = 0$. Similarly as in the case of a contact leaf, $\varpi_{\mathcal O}$ actually agrees with $\varpi$ on $C_{\mathcal O}$. It follows that the imaginary part $\operatorname{Im} \varpi_{\mathcal O}$ agrees with the lcs form $\Omega_{\mathcal O}$. From (\ref{eq:L_pi}) we get
\begin{equation}\label{eq:IL_2}
I^!_{\mathcal O} \mathfrak L = \left\{ (\Delta, \iota_\Delta \varpi_{\mathcal O} + A) : \Delta \in C_{\mathcal O} \otimes \mathbbm C \text{ and } A \in \mathsf{Ann} C_{\mathcal O}  \right\}.
\end{equation}
Using $C_{\mathcal O} \cong T\mathcal O$, we can also think of $\varpi_{\mathcal O}$ as an $L_{\mathcal O} \otimes \mathbbm C$-valued $2$-form on $\mathcal O$. Then, if we denote by $\nabla : T \mathcal O \to D L_{\mathcal O}$ the flat connection in $L_{\mathcal O}$ whose image is $C_{\mathcal O}$, from (\ref{eq:IL_2}) and involutivity, we get $d^\nabla \varpi_{\mathcal O} = 0$. In particular, $d^\nabla \operatorname{Re} \varpi_{\mathcal O} = 0$, and, locally, $\operatorname{Re} \varpi_{\mathcal O} = d^\nabla \eta$, for some $\eta \in \Omega^1 (\mathcal O, L_{\mathcal O}) \subset \Gamma (J^1 L_{\mathcal O})$. Put $B := d_D \eta $. An easy computation shows that
\[
B = \operatorname{Re} \varpi_{\mathcal O} + \mathcal C \wedge \eta
\]
where $\mathcal C : DL_{\mathcal O} \to \mathbbm R_{\mathcal O}$ is the unique $1$-form with kernel $C_{\mathcal O}$, and such that $\langle \mathcal C , \mathbb 1 \rangle = 1$. Hence
\[
(I^!_{\mathcal O} \mathfrak L)^{-B} 
 =  \left\{ (\Delta, \mathrm{i} \cdot \iota_\Delta \operatorname{Im} \varpi_{\mathcal O} + A ) : \Delta \in C_{\mathcal O} \otimes \mathbbm C \text{ and } A \in \mathsf{Ann} C_{\mathcal O} \right\} = \mathfrak L_{\mathcal O}.
\]
For the very last step we used (\ref{eq:L_02}), and the fact that $\operatorname{Im} \varpi_{\mathcal O} = \Omega_{\mathcal O}$ (together with the relationship between $\Omega_\mathcal O$ and $J_{\mathcal O}$ discussed at the beginning of this subsection).
\end{proof}

\begin{proposition}\label{prop:lcs_transversal}
Let $N$ be a transversal at $x_0 \in \mathcal O$ to an even dimensional leaf $\mathcal O$ of $J$. Around $x_0$, the backward image of the complex Dirac-Jacobi structure $\mathfrak L$ along the embedding $I_N : L_N \hookrightarrow L$ is a complex Dirac-Jacobi structure $I^!_N \mathfrak L$ of generalized contact type.
\end{proposition}

\begin{proof}
One can prove that $I^!_N \mathfrak L \subset \mathbbm D L_N$ is a regular subbundle in a very similar way as for the transversal to a contact leaf (proof of Proposition \ref{prop:contact_transversal}) and we leave it to the reader to take care of the obvious adaptations. Now, we show that
\[
\left( I^!_N \mathfrak L \cap I^!_N \overline{\mathfrak L} \right)_{x_0} = 0.
\]
It will follow that $I^!_N \mathfrak L \cap I^!_N \overline{\mathfrak L} = 0$ in a whole neighborhood of $x_0$. So let $(\Delta, \psi) \in (I^!_N \mathfrak L \cap I^!_N \overline{\mathfrak L})_{x_0}$. Then 
\[
\Delta \in p_D \mathfrak L_{x_0} \cap p_D \overline{\mathfrak L}_{x_0} \cap  D_{x_0} L_N \otimes \mathbbm C= ((C_{\mathcal O})_{x_0} \cap D_{x_0} L_N) \otimes \mathbbm C = 0,
\]
and we find $\chi', \chi'' \in J_{x_0}^1 L \otimes \mathbbm C$, such that $(0, \chi') \in \mathfrak L_{x_0}$ (i.e.~$\chi' \in \mathsf{Ann} (p_D \mathfrak L_{x_0})$), $(0, \chi'') \in \overline{\mathfrak L}_{x_0}$ (i.e.~$\chi'' \in \mathsf{Ann} (p_D \overline{\mathfrak L}_{x_0})$), and, additionally, $\psi = I^\ast_N \chi' = I^\ast_N \chi''$. Hence 
\[
\begin{aligned}
\chi' - \chi'' & \in  \mathsf{Ann} \left(p_D \mathfrak L_{x_0} \cap p_D \overline{\mathfrak L}_{x_0}\right) \cap \mathsf{Ann} \left(D_{x_0} L_N \otimes \mathbbm C\right) \\
 & = \mathsf{Ann} \left( (C_{\mathcal O})_{x_0} + D_{x_0} L_N\right) \otimes \mathbbm C = 0.
\end{aligned}
\]
It follows that $(0, \chi') = (0, \chi'') \in \mathfrak L_{x_0} \cap \overline{\mathfrak L}_{x_0} = 0$, so that $\psi = 0$ as well. This concludes the proof.
\end{proof}

\section{Splitting theorems} \label{Sec:splitting}

In this section we prove a local splitting theorem for generalized contact bundles analogous to Weinstein splitting theorem for Poisson structures \cite{W1983}, and similar splitting theorems in Poisson-related geometries: Jacobi geometry \cite{DLM1991}, Dirac geometry \cite{B2014} (see also \cite{DW2008}), Lie algebroid geometry \cite{D2001, F2002, W2000}, generalized complex geometry \cite{AB2006} (see also \cite{Ba2013} for an important refinement of Abouzaid-Boyarchenko result). As expected, our splitting theorem is similar to that for generalized complex manifolds on one side, and to that for Jacobi bundles on the other side. In particular, we actually prove two splitting theorems: one about the local structure around a point in a contact leaf and one about the local structure around a point in a lcs leaf. Our proof is different in spirit from that of Abouzaid and Boyarchenko, and it is rather inspired by the recent work of Bursztyn, Lima and Meinrenken \cite{BLM2016}, who provided a unified approach to splitting theorems in Poisson (and related) geometries.

We begin recalling the splitting theorems of Dazord, Lichnerowicz and Marle for Jacobi bundles \cite{DLM1991}. 

\begin{theorem}\label{theor:Dazord_contact}
Let $(L \to M, J)$ be a Jacobi bundle, and let $N$ be a sufficiently small transversal at $x_0 \in \mathcal O$ to a $(2d + 1)$-dimensional characteristic leaf $\mathcal O$ of $J$. Then, there are
\begin{itemize}
\item[$\triangleright$] a homogenous Poisson structure $(\pi_N, Z_N)$ on $N$, 
\item[$\triangleright$] an open neighborhood $V$ of $0$ in $\mathbbm R^{2d +1}$, and
\item[$\triangleright$] a line bundle isomorphism $\Phi : L \to \mathbbm R_{N \times V}$, covering a diffeomorphism $\phi : M \to N \times V$, locally defined around $x_0$,
\end{itemize}
such that
\begin{enumerate}
\item $\phi$ identifies $N$ with $N \times \{0\}$, and (a neighborhood of $x_0$ in) $\mathcal O$ with $\{x_0\} \times V$,
\item $\Phi$ identifies $J$ with the Jacobi structure $J^\times$ corresponding to the Jacobi pair $(\Lambda^\times, E^\times)$ given by 
\end{enumerate}
\begin{equation}\label{eq:product_contact}
\Lambda^\times = \Lambda_{\mathit{can}} + \pi_N - E_{\mathit{can}} \wedge Z_N, \quad \text{and} \quad E^\times = E_{\mathit{can}},
\end{equation}
where $(\Lambda_{\mathit{can}}, E_{\mathit{can}})$ is the Jacobi pair from Example \ref{ex:J_can}.
\end{theorem}

\begin{remark}
Formula (\ref{eq:product_contact}) has a nice interpretation in terms of Dirac-Jacobi structures. Namely, let $\mathfrak L_{J_{\mathit{can}}} = \operatorname{\mathsf{graph}} J_{\mathit{can}} \subset \mathbbm D \mathbbm R_V$ be the Dirac-Jacobi structure induced by $J_{\mathit{can}}$ on the trivial line bundle, and let $\mathfrak L_N \subset \mathbbm D \mathbbm R_{N}$ be the Dirac-Jacobi structure spanned as follows:
\[
\mathfrak L_N = \left\langle (1-Z_N, 0), (\pi_N^\sharp \eta, \eta + \eta (Z_N) \cdot \mathfrak j ) : \eta \in T^\ast N \right\rangle .
\]
Additionally, let $\mathfrak L^\times = \operatorname{\mathsf{graph}} J^\times \subset \mathbbm D \mathbbm R_{N \times V}$. Then $\mathfrak L^\times$ is the flat product of $\mathfrak L_N$ and $\mathfrak L_{J_{\mathit{can}}}$ with respect to the standard projections $\mathbbm R_{N \times V} \to \mathbbm R_N$, and $\mathbbm R_{N \times V} \to \mathbbm R_V$:
\[
\mathfrak L^\times = \mathfrak L_N \times^! \mathfrak L_{J_{\mathit{can}}}.
\]
 Notice that $\mathfrak{L}_N=\phi|_N^!\mathfrak{L}^\times$. We stress that it is not a coincidence that $\phi|_N^!\mathfrak{L}^\times$ is a Dirac structure coming from a homogeneous Poisson structure. It is proven in \cite{DLM1991} that every transversal to a contact leaf of a Jacobi structure possesses, at least locally, a homogeneous Poisson structure and that different transversals possess isomorphic homogeneous Poisson structures. A similar statement holds for Dirac-Jacobi structures \cite[Proposition 6.9]{V2015}. In a similar way, every transversal to a lcs leaf possesses a Jacobi structure (see \cite{DLM1991} again, and \cite[Proposition 6.9]{V2015} for the Dirac-Jacobi case).
\end{remark}

\begin{theorem}\label{theor:Dazord_lcs}
Let $(L \to M, J)$ be a Jacobi bundle, and let $N$ be a sufficiently small transversal at $x_0 \in \mathcal O$ to a $2d$-dimensional characteristic leaf $\mathcal O$ of $J$. Then, there are
\begin{itemize}
\item[$\triangleright$] a Jacobi pair $(\Lambda_N, E_N)$ on $N$, 
\item[$\triangleright$] an open neighborhood $V$ of $0$ in $\mathbbm R^{2d}$, and
\item[$\triangleright$] a line bundle isomorphism $\Phi : L \to \mathbbm R_{N \times V}$, covering a diffeomorphism $\phi : M \to N \times V$, locally defined around $x_0$,
\end{itemize}
such that
\begin{enumerate}
\item $\phi$ identifies $N$ with $N \times \{0\}$, and (a neighborhood of $x_0$ in) $\mathcal O$ with $\{x_0\} \times V$,
\item $\Phi$ identifies $J$ with the Jacobi structure $J^\times$ corresponding to the Jacobi pair $(\Lambda^\times, E^\times)$ given by 
\end{enumerate}
\begin{equation}\label{eq:product_lcs}
\Lambda^\times = \Lambda_{N} + \pi_{\mathit{can}} - E_{N} \wedge Z_{\mathit{can}}, \quad \text{and} \quad E^\times = E_{N}, 
\end{equation}
where $(\pi_{\mathit{can}}, Z_{\mathit{can}})$ is the homogeneous Poisson structure from Example \ref{ex:can_hP}.
\end{theorem}

\begin{remark}
Again, formula (\ref{eq:product_lcs}) has an interpretation in terms of Dirac-Jacobi structures. Namely, let $\mathfrak L_{(\pi_{\mathit{can}}, Z_{\mathit{can}})} \subset \mathbbm D \mathbbm R_{V}$ be the Dirac-Jacobi structure spanned as follows:
\[
\mathfrak L_{(\pi_{\mathit{can}}, Z_{\mathit{can}})}  = \left\langle (1-Z_{\mathit{can}}, 0), (\pi_{\mathit{can}}^\sharp \eta, \eta + \eta (Z_{\mathit{can}}) \cdot \mathfrak j ) : \eta \in T^\ast V \right\rangle
\]
and let $\mathfrak L_{N} = \operatorname{\mathsf{graph}} J_{N} \subset \mathbbm D \mathbbm R_N$ be the Dirac-Jacobi structure induced by $J_{N}$. Finally, let $\mathfrak L^\times = \operatorname{\mathsf{graph}} J^\times \subset \mathbbm D \mathbbm R_{N \times V}$. Then $\mathfrak L^\times$ is the flat product of $ \mathfrak L_{(\pi_{\mathit{can}}, Z_{\mathit{can}})}$ and $\mathfrak L_N$ with respect to the standard projections $\mathbbm R_{N \times V} \to \mathbbm R_N$, and $\mathbbm R_{N \times V} \to \mathbbm R_V$:
\[
\mathfrak L^\times = \mathfrak L_{(\pi_{\mathit{can}}, Z_{\mathit{can}})}\times^! \mathfrak L_N .
\]
\end{remark}

\subsection{Splitting around a contact point}

 We are finally ready to prove our main results. We begin with a remark.
 
 \begin{remark}\label{rem:can_odd}
The Jacobi structure $J_{\mathit{can}}$ from Example \ref{ex:J_can} is non degenerate. Hence it corresponds to a contact structure $H_{\mathit{can}}$. Namely, let $\omega_{\mathit{can}} = J_{\mathit{can}}^{-1}$ be the Atiyah $2$-form inverting $J_{\mathit{can}}$. Then 
\[
\omega_{\mathit{can}} = dx^i \wedge dp_i - (du - p_i dx^i) \wedge \mathfrak j,
\]
and $\iota_{\mathbb 1} \omega_{\mathit{can}}$ agrees with
\[
\theta_{\mathit{can}} = du - p_i dx^i,
\]
the canonical contact $1$-form on $\mathbbm R^{2d +1}$, and $H_{\mathit{can}} = \ker \theta_{\mathit{can}}$ is the canonical contact structure on $\mathbbm R^{2d +1}$. The latter can be equivalently encoded in a generalized contact structure
\begin{equation}\label{eq:I_can}
\left( 
\begin{array}{cc}
0 & J_{\mathit{can}}^\sharp \\
-(\omega_{\mathit{can}})_\flat & 0
\end{array}
\right)
\end{equation}
whose $+\mathrm{i}$-eigenbundle is
\[
\mathfrak L_{\mathit{can}}^{\mathit{odd}} = \left\{ (J^\sharp_{\mathit{can}} (\psi), \mathrm{i} \psi) : \psi \in J^1 \mathbbm R_{\mathbbm R^{2d + 1}} \otimes \mathbbm C \right\}.
\]
Clearly, we also have
\begin{equation}\label{eq:L_can_DR}
\mathfrak L_{\mathit{can}}^{\mathit{odd}} = \left\{ (\Delta, \mathrm{i} \cdot \iota_{\Delta}\omega_{\mathit{can}}) : \Delta \in D\mathbbm R_{\mathbbm R^{2d + 1}} \otimes \mathbbm C\right\} = (D\mathbbm R_{\mathbbm R^{2d + 1}} \otimes \mathbbm C)^{\mathrm{i} \omega_{\mathit{can}}},
\end{equation}
i.e.~$\mathfrak L_{\mathit{can}}^{\mathit{odd}} $ can be seen as the \emph{complex} $B$-field transformation of the complex Dirac-Jacobi structure $D\mathbbm R_{\mathbbm R^{2d + 1}} \otimes \mathbbm C \subset \mathbbm D \mathbbm R_{\mathbbm R^{2d + 1}}\otimes \mathbbm C$ by means of the closed \emph{complex} Atiyah $2$-form $\mathrm{i} \omega_{\mathit{can}}$. This simple remark will be useful below. Actually, similar considerations hold for any non-degenerate Jacobi structure. Notice, however, that (\ref{eq:L_can_DR}) does \emph{not} mean that there is a Courant-Jacobi automorphism intertwining (\ref{eq:I_can}) with some other generalized contact structure. Yet in other words, $D\mathbbm R_{\mathbbm R^{2d + 1}} \otimes \mathbbm C$ is \emph{not} a complex Dirac-Jacobi structure of generalized contact type, and the obvious reason is that only \emph{real} $B$-field transformations are Courant-Jacobi automorphisms, while $\mathrm{i} \omega_{\mathit{can}}$ is a purely imaginary Atiyah $2$-form.
 \end{remark}
 
 \begin{theorem}\label{theor:splitting_contact}
Let $(L \to M, \mathbbm K)$ be a generalized contact bundle, let $\mathfrak L \subset \mathbbm D L\otimes \mathbbm C$ be the $+\mathrm{i}$-eigenbundle of $\mathbbm K$, and let $N$ be a sufficiently small transversal at $x_0 \in \mathcal O$ to a $(2d + 1)$-dimensional characteristic leaf $\mathcal O$. Then, there are
\begin{itemize}
\item[$\triangleright$] an open neighborhood $U$ of $0$ in $\mathbbm R^{2d +1}$,
\item[$\triangleright$] a line bundle isomorphism $\Phi : L \to \mathbbm R_{N \times U}$, covering a diffeomorphism $\phi : M \to N \times U$, locally defined around $x_0$, and
\item[$\triangleright$] a closed Atiyah $2$-form $B$ on $\mathbbm R_{N \times U}$ 
\end{itemize}
such that
\begin{enumerate}
\item $\phi$ identifies $N$ with $N \times \{0\}$, and (a neighborhood of $x_0$ in) $\mathcal O$ with $\{x_0\} \times U$,
\item the Courant-Jacobi automorphism $e^B \circ \mathbbm D \Phi$ identifies $\mathfrak L$ with
\end{enumerate}
\begin{equation}
\mathfrak L_N \times^! \mathfrak L_{\mathit{can}}^{\mathit{odd}},
\end{equation}
the flat product of $\mathfrak L_N$ and $\mathfrak L_{\mathit{can}}^{\mathit{odd}}$ with respect to the standard projections $P_N : \mathbbm R_{N \times U} \to \mathbbm R_N$, and $P_U : \mathbbm R_{N \times U} \to \mathbbm R_U$. Here $\mathfrak L_N = I_N^! \mathfrak L$ is the complex Dirac-Jacobi structure of homogeneous generalized complex type induced by $\mathfrak L$ on $N$ (see Proposition \ref{prop:contact_transversal}), and $\mathfrak L_{\mathit{can}}^{\mathit{odd}}$ is the complex Dirac-Jacobi structure of generalized contact type from Remark \ref{rem:can_odd}.
\end{theorem}

\begin{proof}
The present proof and, similarly, the proof of Theorem \ref{theor:splitting_lcs} below, are inspired by a general technique recently proposed by Bursztyn, Lima and Meinrenken to prove splitting theorems in Poisson and related geometries. Without loss of generality, we can assume that $M = N \times V$, $L = \mathbbm R_{N \times V}$ is the trivial line bundle, and the Jacobi structure underlying $\mathbbm K$ is $J^\times$, where $V$, $N$ and $J^\times$ are as in Theorem \ref{theor:Dazord_contact}. Now let $\psi \in \Gamma (J^1 \mathbbm R_{N \times V})$ be given by
\[
\psi =  x^i dp_i - p_i dx^i + \left(x^i p_i -u  \right) \cdot \mathfrak j.
\]
Put $\mathcal E := J^\sharp \psi$. Then
\begin{equation}\label{eq:Euler}
\mathcal E = x^i \frac{\partial}{\partial x^i} + u \frac{\partial}{\partial u} + p_i \frac{\partial}{\partial p_i}
\end{equation}
is the Euler vector field on $V$. More precisely, it is the covariant derivative along the Euler vector field with respect to the canonical flat connections in $\mathbbm R_{N \times V}$. By changing $\psi$ into $f \psi$ with $f \in C^\infty (V \times N)$ a suitable bump function equal to $1$ around $N$, we can arrange that $\mathcal E$ is complete, while (\ref{eq:Euler}) still holds around $N$. Denote by $\{ \Phi_t \}$ the flow of $\mathcal E$ on $\mathbbm R_{N \times U}$, and let $\{ \phi_t \}$ be its projection to $N \times V$. Then, for all $t \leq 0$ we have
\[
\Phi_t (x, v\,; r) = (x, \mathrm{exp} (t) \cdot v\,; r), \quad (x, v\,; r) \in N \times V \times \mathbbm R,
\]
at least when $v$ is small enough. Put
\[
U := \left\{ v \in V : \lim_{t \to - \infty} \phi_t (x, v) \in N \times \{0\} \text{ for all $x \in N$}\right\}.
\]
Then $U \subset V$ is an open subset and $\mathcal E$ remains complete when restricted to $U$. Additionally, the family of maps
\[
K_s := \Phi_{\log (s)} : \mathbbm R_{N \times U} \to \mathbbm R_{N \times U}
\]
extends smoothly to $s = 0$, and $K_0 = I_N \circ P_N$, where $P_N : \mathbbm R_{N \times U} \to \mathbbm R_N$ is the canonical projection and $I_N : \mathbbm R_N \to \mathbbm R_{N \times U} $ is the embedding at $v = 0 \in U$. 

Now consider the $+\mathrm{i}$-eigensection $\alpha$ of $\mathbbm K$ given by
\begin{equation}\label{eq:X}
\mathrm{i} (0, \psi) + \mathbbm K (0, \psi) = (\mathcal E, \chi) \in \Gamma (\mathfrak L),
\end{equation}
where we put $\chi := \mathrm{i} \psi - \varphi^\dag \psi$. Consider also the infinitesimal Courant-Jacobi automorphism 
\begin{equation}\label{eq:XX}
(\mathcal L_{\mathcal E} - \overline{d_D \chi}, \mathcal E) = \left(\blq (\mathcal E, \chi), - \brq, \mathcal E \right).
\end{equation}
 From (\ref{eq:X}), (\ref{eq:XX}), and involutivity, the flow of $(\mathcal L_{\mathcal E} - \overline{d_D \chi}, \mathcal E)$ preserves $\mathfrak L$. From Remark \ref{rem:flow} this flow is
\[
\{ (e^{C_t} \circ \mathbbm D \Phi_t, \Phi_t) \}, \quad \text{where } C_t = \int_{0}^t \Phi_{-\epsilon}^\ast (d_D \chi) d \epsilon .
\]
In particular
\begin{equation}\label{eq:L}
\mathfrak L = (\mathbbm D\Phi_{- \log (s)} \mathfrak L)^{C_{- \log (s)}} = (K_s^! \mathfrak L)^{C_{-\log (s)}},
\end{equation}
for all $s > 0$. Put $B_s := C_{- \log (s)}$ and compute
\[
\begin{aligned}
B_s & = \int_{0}^{- \log (s)} \Phi_{- \epsilon}^\ast (d_D \chi) d \epsilon = \int_s^1 \tau^{-1} K_\tau^\ast (d_D \chi) d\tau \\
& = \mathrm{i} \int_s^1 \tau^{-1} K_\tau^\ast (d_D \psi) d\tau - \int_s^1 \tau^{-1} K_\tau^\ast (d_D \varphi^\dag \psi) d\tau.
\end{aligned}
\]
In a possibly smaller neighborhood of $N \times \{0\}$ we have
\[
 K_\tau^\ast (d_D \psi) = K_\tau^\ast \left(2 dx^i \wedge dp_i + (2p_i dx^i - du) \wedge \mathfrak j \right) = 2\tau^2 dx^i \wedge dp_i + (2\tau^2 p_i dx^i - \tau du) \wedge \mathfrak j,
\]
for all $\tau \in [0,1]$. Hence, for all $s \in (0, 1]$,
\[
\int_s^1 \tau^{-1} K_\tau^\ast (d_D \psi) d\tau = (1-s) \left((1+s)dx^i \wedge dp_i - (du - (1+s)p_i dx^i) \wedge \mathfrak j \right),
\]
which extends to $s = 0$. We conclude that, in a possibly smaller neighborhood of $N \times \{0\}$, $B_0$ is well-defined, and, more precisely,
\[
B_0 = B + \mathrm{i} \omega_{\mathit{can}} 
\] 
where $B $ is a certain real closed Atiyah $2$-form. Finally, from (\ref{eq:L}), by continuity, we get, in a neighborhood of $N \times \{0\}$
\[
\begin{aligned}
\mathfrak L & = (K_0^! \mathfrak L)^{B_0} \\
                   & = (P_N^! I_N^! \mathfrak L)^{B + \mathrm{i} \omega_{\mathit{can}}} \\
                   & = ( P_N^! \mathfrak L_N \star D \mathbbm R_{N \times U} \otimes \mathbbm C)^{B + \mathrm{i} \omega_{\mathit{can}}} \quad \text{(Remark \ref{rem:star_neutral})} \\
                   & = (P_N^! \mathfrak L_N \star P_U^! (D \mathbbm R_{U} \otimes \mathbbm C))^{B + \mathrm{i} \omega_{\mathit{can}}} \\
                   & = (\mathfrak L_N \times^! D \mathbbm R_{U}\otimes \mathbbm C)^{B + \mathrm{i} \omega_{\mathit{can}}} \\
                   & = (\mathfrak L_N \times^! (D \mathbbm R_{U}\otimes \mathbbm C)^{\mathrm{i} \omega_{\mathit{can}}})^B  \quad \text{(Remark \ref{rem:product_B-field})}\\
                   & = (\mathfrak L_N \times^! \mathfrak L_{\mathit{can}}^{\mathit{odd}})^B \quad \text{(Equation (\ref{eq:L_can_DR}))}.
\end{aligned}
\]
\end{proof}

\subsection{Splitting around an lcs point}

\begin{remark}\label{rem:can_even}
Consider the homogeneous Poisson structure $(\pi_{\mathit{can}}, Z_{\mathit{can}})$ from Example \ref{ex:can_hP}. The Poisson structure $\pi_{\mathit{can}}$ is non-degenerate and its inverse is $\Omega_{\mathit{can}} = dx^i \wedge dp_i$, the canonical symplectic structure on $\mathbbm R^{2d}$. In its turn $\Omega_{\mathit{can}} = - d \Theta_{\mathit{can}}$, where
\[
\Theta_{\mathit{can}} = p_i dx^i
\]
is the Liouville $1$-form. The pair $(\Omega_{\mathit{can}}, Z_{\mathit{can}})$ is a \emph{homogeneous symplectic structure} in the sense that $\mathcal L_{Z_{\mathit{can}}}\Omega_{\mathit{can}} = \Omega_{\mathit{can}}$, and we can encode it in a complex Dirac-Jacobi structure of homogeneous generalized complex type
\[
\mathfrak L_{\mathit{can}}^{\mathit{ev}} := \left\langle (1-Z_{\mathit{can}}, 0), (\pi_{\mathit{can}}^\sharp \eta, \mathrm{i} \cdot \eta) : \eta \in T^\ast \mathbbm R^{2d} \otimes \mathbbm C \right\rangle .
\]
Now, consider the exact Atiyah $2$-form $\xi_{\mathit{can}} = - d_D \Theta_{\mathit{can}}$. It is easy to see that
\[
\mathfrak L_{\mathit{can}}^{\mathit{ev}} = \left\{ (\Delta, \mathrm{i} \cdot \iota_\Delta \xi_{\mathit{can}}) : \Delta \in D \mathbbm R_{\mathbbm R^{2d}} \otimes \mathbbm C \right\} = (D\mathbbm R_{\mathbbm R^{2d}} \otimes \mathbbm C)^{\mathrm{i} \xi_{\mathit{can}}},
\]
i.e.~$\mathfrak L_{\mathit{can}}^{\mathit{ev}} $ is the complex $B$-field transformation of the complex Dirac-Jacobi structure $D\mathbbm R_{\mathbbm R^{2d}} \otimes \mathbbm C \subset \mathbbm D \mathbbm R_{\mathbbm R^{2d}}\otimes \mathbbm C$ by means of the complex closed Atiyah $2$-form $\mathrm{i} \xi_{\mathit{can}}$. Similar considerations hold for any homogeneous Poisson structure $(\pi, Z)$ such that $\pi$ is non-degenerate. We leave the simple details to the reader.
 \end{remark}

  \begin{theorem}\label{theor:splitting_lcs}
Let $(L \to M, \mathbbm K)$ be a generalized contact bundle, let $\mathfrak L \subset \mathbbm D L\otimes \mathbbm C$ be the $+\mathrm{i}$-eigenbundle of $\mathbbm K$, and let $N$ be a sufficiently small transversal at $x_0 \in \mathcal O$ to a $2d$-dimensional characteristic leaf $\mathcal O$. Then, there are
\begin{itemize}
\item[$\triangleright$] an open neighborhood $U$ of $0$ in $\mathbbm R^{2d}$,
\item[$\triangleright$] a line bundle isomorphism $\Phi : L \to \mathbbm R_{N \times U}$, covering a diffeomorphism $\phi : M \to N \times U$, locally defined around $x_0$, and
\item[$\triangleright$] a closed Atiyah $2$-form $B$ on $\mathbbm R_{N \times U}$ 
\end{itemize}
such that
\begin{enumerate}
\item $\phi$ identifies $N$ with $N \times \{0\}$, and $\mathcal O$ with $\{x_0\} \times U$,
\item the Courant-Jacobi automorphism $e^B \circ \mathbbm D \Phi$ identifies $\mathfrak L$ with
\end{enumerate}
\begin{equation}
\mathfrak L_N \times^! \mathfrak L_{\mathit{can}}^{\mathit{ev}},
\end{equation}
the flat product of $\mathfrak L_N$ and $\mathfrak L_{\mathit{can}}^{\mathit{ev}}$ with respect to the standard projections $P_N : \mathbbm R_{N \times U} \to \mathbbm R_N$, and $P_U : \mathbbm R_{N \times U} \to \mathbbm R_U$. Here $\mathfrak L_N = I_N^! \mathfrak L$ is the complex Dirac-Jacobi bundle structure of generalized contact type induced by $\mathfrak L$ on $N$ (see Proposition \ref{prop:lcs_transversal}), and $\mathfrak L_{\mathit{can}}^{\mathit{ev}}$ is the complex Dirac-Jacobi structure of homogeneous generalized complex type from Remark \ref{rem:can_even}.
\end{theorem}

\begin{proof}
Without loss of generality, we assume that $M = N \times V$, $L = \mathbbm R_{N \times V}$ is the trivial line bundle, and the Jacobi structure underlying $\mathbbm K$ is $J^\times$, where $V$, $N$ and $J^\times$ are as in Theorem \ref{theor:Dazord_lcs}. Let $\psi \in \Gamma (J^1 \mathbbm R_{N \times V})$ be given by
\[
\psi = x^i dp_i - p_i dx^i + \left(x^i p_i \right) \cdot \mathfrak j.
\]
So that
\begin{equation*}
\mathcal E := J^\sharp \psi =  x^i \frac{\partial}{\partial x^i} + p_i \frac{\partial}{\partial p_i}
\end{equation*}
is the Euler vector field on $V$. We define $U$, $K_0$, $B_0$ exactly as in the proof of theorem \ref{theor:splitting_contact}. A direct computation then shows that $B_0$ is well defined around $N \times \{0\}$ and it is given by
\[
B_0 = B + \mathrm{i} \xi_{\mathit{can}}
\]
for some real closed Atiyah $2$-form $B$. Exactly as in the proof of Theorem \ref{theor:splitting_contact} we now get
\[
\mathfrak L = (K_0^! \mathfrak L)^{B_0} = (\mathfrak L_N \times^! \mathfrak L_{\mathit{can}}^{ev})^B.
\]
\end{proof}

\section{The regular case}\label{Sec:regular}

Let $(L \to M, \mathbbm K)$ be a generalized contact bundle, and let $J$ be the Jacobi structure underlying $\mathbbm K$. A point $x_0 \in M$ is a \emph{regular point} for $\mathbbm K$ if the characteristic leaves of $J$ has constant dimension around $x_0$. Similarly as in the generalized complex case \cite{G2011}, when $x_0 \in M$ is a regular point, the Splitting Theorems \ref{theor:splitting_contact} and \ref{theor:splitting_lcs} simplify and we get honest local normal form theorems around $x_0$.

\subsection{Local normal form around a regular contact point}

\begin{remark}
Denote by $A_{\mathit{can}}$ the standard complex structure on $\mathbbm C^n$. It can be encoded in a generalized complex structure
\begin{equation}\label{eq:can_complex}
\left( 
\begin{array}{cc}
A_{\mathit{can}} & 0 \\
0 & - A_{\mathit{can}}^\ast
\end{array}
\right)
\end{equation}
whose $+\mathrm{i}$-eigenbundle is $T^{1,0} \mathbbm C^n \oplus (T^{0,1} \mathbbm C^n)^\ast$. The generalized complex structure (\ref{eq:can_complex}) is homogeneous with respect to the zero section $(0,0) \in \Gamma (\mathbbm T \mathbbm C^n)$, and we get the following complex Dirac-Jacobi structure of homogeneous generalized complex type on $\mathbbm R_{\mathbbm C^n}$:
\[
\mathfrak L_{\mathbbm C^n} := \left\langle (1, 0), (X, \eta) : X \in T^{1,0} \mathbbm C^n, \text{ and }\eta \in (T^{0,1} \mathbbm C^n)^\ast \right\rangle.
\]  
\end{remark}

\begin{theorem}
Let $(L \to M, \mathbbm K)$ be a generalized contact bundle with $\dim M = 2 (n+d) +1$. Let $\mathfrak L \subset \mathbbm D L\otimes \mathbbm C$ be the $+\mathrm{i}$-eigenbundle of $\mathbbm K$, and let $x_0 \in M$ be a regular point in a $(2d + 1)$-dimensional characteristic leaf. Then, locally, around $x_0$, $\mathfrak L$ is isomorphic to the flat product
\[
\mathfrak L_{\mathbbm C^n} \times^! \mathfrak L_{\mathit{can}}^{\mathit{odd}}
\]
with respect to the standard projections $\mathbbm R_{\mathbbm C^n \times \mathbbm R^{2d +1}} \to \mathbbm R_{\mathbbm C^n}$, and $\mathbbm R_{\mathbbm C^n \times \mathbbm R^{2d +1}} \to \mathbbm R_{\mathbbm R^{2d +1}}$, up to a $B$-field transformation.
\end{theorem}

\begin{proof}
Let $N$ be a sufficiently small transversal to the characteristic leaf through $x_0$. 
From Theorem \ref{theor:splitting_contact}, it is enough to show that the induced 
Dirac-Jacobi structure $\mathfrak L_N = I^!_N \mathfrak L$ on $N$ is isomorphic 
to $\mathfrak L_{\mathbbm C^n}$ around $x_0$ up to a $B$-field transformation. 
From the proof of Proposition \ref{prop:contact_transversal} and the fact that $x_0$
is a regular point, it easily follows that $\mathfrak L_N \cap \overline{\mathfrak L_N}$ is (everywhere, not only at $x_0$) spanned by a section of the form $(\mathbb 1, \zeta)$, with $\zeta \in T^\ast N \otimes L_N$, and, from the proof of Proposition \ref{prop:hom_gen_compl}, $\mathfrak L_N$ is isomorphic to a Dirac-Jacobi structure of generalized complex type of the form $\mathfrak L_{(\mathbbm J, \mathbbm Z)}$ (see (\ref{eq:hgc_DJ})) with $Z = 0$. In particular, 
\begin{enumerate}
\item $\pi = 0$, 
\item $A$ is a complex structure on $N$, and
\item $\sigma = \iota_A d \zeta$
\end{enumerate}
(see Definition \ref{def:hom_gen_compl}). A direct computation exploiting (1) and (3) shows that, after the $B$-field transformation $(e^{d_D \zeta} , \mathrm{id})$, we achieve $\zeta = \sigma = 0$. Finally, with a diffeomorphism, we achieve $A = A_{\mathit{can}}$, showing that $\mathfrak L_N$ is isomorphic to $\mathfrak L_{\mathbbm C^n}$ up to a $B$-field transformation.
\end{proof}

\subsection{Local normal form around a regular lcs point}

\begin{remark}
Consider the standard complex structure $\varphi_{\mathit{can}}$ on the gauge algebroid of the trivial line bundle over the cylinder $\mathbbm R \times \mathbbm C^n$ from Example \ref{ex:varphi_can} in the Appendix. It can be encoded in a generalized contact structure
\begin{equation}\label{eq:can_Rxcomplex}
\left( 
\begin{array}{cc}
\varphi_{\mathit{can}} & 0 \\
0 & - \varphi_{\mathit{can}}^\dag
\end{array}
\right)
\end{equation}
whose $+\mathrm{i}$-eigenbundle is 
\[
\mathfrak L_{\mathbbm R \times \mathbbm C^n} = D^{1,0} \mathbbm R_{\mathbbm R \times \mathbbm C^n} \oplus (D^{0,1} \mathbbm R_{\mathbbm R \times \mathbbm C^n})^\ast
\]
(see Appendix \ref{appendix} for more details).
\end{remark}

\begin{theorem}
Let $(L \to M, \mathbbm K)$ be a generalized contact bundle with $\dim M = 2 (n+d) +1$. Let $\mathfrak L \subset \mathbbm D L\otimes \mathbbm C$ be the $+\mathrm{i}$-eigenbundle of $\mathbbm K$, and let $x_0 \in M$ be a regular point in a $2d$-dimensional characteristic leaf. Then, locally, around $x_0$, $\mathfrak L$ is isomorphic to the flat product
\[
\mathfrak L_{\mathbbm R \times \mathbbm C^n} \times^! \mathfrak L_{\mathit{can}}^{\mathit{ev}}
\]
with respect to the standard projections $\mathbbm R_{(\mathbbm R \times \mathbbm C^n) \times \mathbbm R^{2d}} \to \mathbbm R_{\mathbbm R \times \mathbbm C^n}$, and $\mathbbm R_{(\mathbbm R \times \mathbbm C^n) \times \mathbbm R^{2d}} \to \mathbbm R_{\mathbbm R^{2d}}$, up to a $B$-field transformation.
\end{theorem}

\begin{proof}
Let $N$ be a sufficiently small transversal to the characteristic leaf through $x_0$. 
From Theorem \ref{theor:splitting_lcs}, it is enough to show that the induced 
Dirac-Jacobi structure $\mathfrak L_N = I^!_N \mathfrak L$ on $N$ is isomorphic 
to $\mathfrak L_{\mathbbm R \times \mathbbm C^n}$ around $x_0$ up to a $B$-field transformation. 
As $x_0$ is a regular point, the characteristic foliation $\mathcal F$ is a regular lcs foliation around $x_0$. In particular $p_D \mathfrak L \cap p_D \overline{\mathfrak L} = \operatorname{im} \nabla_{\mathcal F}$ where $\nabla_{\mathcal F}$ is a flat leaf-wise connection along $\mathcal F$ in $L$. Hence
 $p_D \mathfrak L_N \cap p_D \overline{\mathfrak L_N} = p_D \mathfrak L \cap p_D \overline{\mathfrak L} \cap D L_N = \operatorname{im} \nabla_{\mathcal F} \cap D L_N = 0$. This means that $\mathfrak L_N$ is the $+\mathrm{i}$-eigenbundle of a generalized contact structure $\mathbbm K_N$ of the form
 \begin{equation}\label{eq:I_N}
\mathbbm K_N =  \left(
 \begin{array}{cc}
 \varphi_N & 0 \\
 (\omega_N)_\flat & - \varphi_N^\dag
 \end{array}
 \right).
 \end{equation}
 In particular $\varphi_N$ is an integrable complex structure on the Atiyah algebroid $DL_N$. In the following we refer to the Appendix for notation and the main properties of such a complex structure. First of all, notice that from (\ref{eq:I_N}) we have $p_D \mathfrak L_N = D^{(1,0)} L_N$ (recall from the Appendix that $D^{(1,0)} L_N$ denotes the $+\mathrm{i}$-eigenbundle of $\varphi_N$). Define a complex Atiyah $2$-form $\gamma \in \Omega^{(2,0)}_{L_N} \otimes \mathbbm C$ by putting
 \[
 \gamma (\Delta, \nabla) = \langle \psi, \nabla \rangle, \quad \Delta, \nabla \in D^{(1,0)}L_N,
 \]
 where $\psi \in J^1 L \otimes \mathbbm C$ is any element such that $(\Delta, \psi) \in \mathfrak L_N$. A straightforward computation using the involutivity of $\mathfrak L_N$ shows that $\partial_D \gamma = 0$, and, from Remark \ref{rem:Dolbeault}, locally around $x_0$, there is $\rho \in \Omega^{(1,0)}_{L_N}$ such that $\gamma = \partial_D \rho$. It is easy to see that
 \[
B := - 2 \operatorname{Re} \left(\gamma + \overline \partial_D \rho \right) 
 \]
 is a (real) closed Atiyah $2$-form. We claim that
 \begin{equation}\label{eq:finale}
 \mathfrak L_N^B = D^{(1,0)} L_N \oplus \mathsf{Hom}(D^{(0,1)} L_N, L_N).
 \end{equation}
 This follows, after a simple computation, from the remark that
 \[
 \mathfrak L_N = \operatorname{\mathsf{graph}} \gamma \oplus \mathsf{Hom}(D^{(0,1)} L_N, L_N)
 \]
 and the fact that $(\gamma + B)_\flat$ takes values in $\mathsf{Hom}(D^{(0,1)} L_N, L_N)$.
 
Notice that (\ref{eq:finale}) means that $\mathfrak L_N^B$ is the $+\mathrm{i}$-eigenbundle of the generalized contact structure
 \[
 \left(
 \begin{array}{cc}
 \varphi_N & 0 \\
 0 & - \varphi_N^\dag
 \end{array}
 \right).
 \]
 Finally, in view of Theorem \ref{theor:ANN}, with a line bundle isomorphism we can achieve $\varphi_N = \varphi_{\mathit{can}}$, and this concludes the proof.
 \end{proof}
 
\noindent \textbf{Acknowledgments.} We thank both referees for having read our first manuscript carefully, and for numerous comments that helped us improving the presentation significantly. LV is member of GNSAGA of INdAM.

\appendix

\section{Complex structures on the gauge algebroid}\label{appendix}

Let $L \to M$ be a line bundle. In this appendix we study the local properties of a generalized contact structure of \emph{complex type}, i.e.~a generalized contact structure $\mathbbm K$ on $L$, of the form
\begin{equation}\label{eq:gcs_ct}
\mathbbm K =
\left(
\begin{array}{cc}
\varphi & 0 \\
0 & -\varphi^\dag
\end{array}
\right).
\end{equation}
In this case $\varphi : DL \to DL$ is a(n integrable) \emph{complex structure} on the gauge algebroid $DL$, i.e.
\begin{itemize}
\item[$\triangleright$] $\varphi$ is \emph{almost complex}, i.e.~$\varphi^2 = - \mathrm{id}$,
\item[$\triangleright$] $\varphi$ is \emph{integrable}, i.e.~its Lie algebroid \emph{Nijenhuis torsion} $\mathcal N_\varphi$ vanishes.
\end{itemize}
Here $\mathcal N_\varphi : \wedge^2 DL \to DL$ is the skew-symmetric bilinear map defined by
\[
\mathcal N_\varphi (\Delta, \nabla) = [\varphi \Delta, \varphi \nabla] - [\Delta, \nabla] - \varphi \left( [\varphi \Delta, \nabla] + [\Delta, \varphi \nabla] \right), \quad \Delta, \nabla \in \Gamma(DL).
\]
Conversely, given a complex structure on $DL$, (\ref{eq:gcs_ct}) defines a generalized contact structure.

\begin{example}\label{ex:varphi_can}
Consider the cylinder $\mathbbm R \times \mathbbm C^n$ over the standard complex space $\mathbbm C^n$. Let $u$ be the standard real coordinate on the first factor, and let $z^i = x^i + \mathrm{i} y^i$, $i = 1, \ldots, n$, be the standard complex coordinates on the second factor. There is a canonical integrable complex structure $\varphi_{\mathit{can}}$ on the gauge algebroid of the trivial line bundle $\mathbbm R_{\mathbbm R \times \mathbbm C^n}$ defined by
\[
\varphi_{\mathit{can}} \mathbb 1 = \frac{\partial}{\partial u}, \quad \text{and} \quad \varphi_{\mathit{can}} \frac{\partial}{\partial x^i} = \frac{\partial}{\partial y^i}.
\]
\end{example}

\begin{example}[Normal Almost Contact Structures]\label{exam:nacs}
Our main reference for this example is \cite{IW2005}, where the reader will find basically all the proofs. We will see in this example and Lemma \ref{lem:ac} that almost contact structures (resp.~normal almost contact structures) are locally the same as almost complex structures (resp.~integrable almost complex structures) on the gauge algebroid of a trivial line bundle $\mathbbm R_M \to M$. Recall that an \emph{almost contact structure} on a manifold $M$ is a triple $(\Phi, \xi, \eta)$, where $\Phi : TM \to TM$ is a $(1,1)$-tensor, $\xi$ is a vector field, and $\eta$ is a $1$-form on $M$ such that
\[
	\Phi^2=-\mathrm{id} +\eta\otimes \xi, \quad \Phi(\xi)=0,\quad \eta\circ\Phi=0, \quad \text{and} \quad \eta(\xi)=1.
\]
See, e.g., \cite{B2002} for more details. The idea behind this definition is that \emph{an almost contact structure is the odd-dimensional analogue of an almost complex structure}. We believe that the use of line bundles and their gauge algebroids makes the analogy much more transparent.  Namely, recall that the gauge algebroid of the trivial line bundle is $D \mathbbm R_M \cong TM \oplus \mathbbm R_M$. Now take a triple $(\Phi, \xi, \eta)$ consisting of an $(1,1)$-tensor, a vector field and a $1$-form on $M$, and let $\varphi :  D\mathbbm{R}_M  \to D\mathbbm{R}_M$ be the endomorphism given by 
	\begin{equation}\label{Ex:AlmNorCon}
	\varphi (X,r) = (\Phi(X)-r\xi, \eta(X))
	\end{equation}
Then $(\Phi, \xi, \eta)$ is an almost contact structure if and only if $\varphi$ is an almost complex structure, i.e.~$\varphi^2=-\mathrm{id}$. Additionally, $\varphi$ is integrable if and only if   
	\begin{equation}\label{eq:normal_ac}
	\mathcal N_\Phi+d\eta\otimes\xi=0,\quad d\eta(\Phi-,-)+d\eta(-,\Phi-)=0,\quad \mathcal{L}_\xi \Phi=0, \quad \text{and} \quad \mathcal{L}_{\xi}\eta=0,
	\end{equation}
where $\mathcal N_\Phi$ is the Nijenhuis torsion of $\Phi$ \cite{IW2005}. One can actually show that the first condition in (\ref{eq:normal_ac}) implies the other ones \cite[Section 6.1]{B2002} (see also \cite{IW2005}). An almost contact structure $(\Phi,\xi,\eta)$ such that $\mathcal N_\Phi +d\eta\otimes \xi = 0$ is called  
\emph{normal} \cite{B2002}. So normal almost contact structures provide examples of complex structures on the Atiyah algebroid (of the trivial line bundle), and, in turn, of generalized contact structures of complex type. It turns out that, locally, every generalized contact structure of complex type is of this form (see Lemma \ref{lem:ac} below).
\end{example}

Example \ref{exam:nacs} is special in view of the following

\begin{lemma}\label{lem:ac}
Let $L\to M$ be a line bundle and let $\varphi : DL \to DL$ be an integrable complex structure. Then, around every point of $M$,
there is a trivialization $L \cong \mathbbm R_M$ identifying $\varphi$ with a complex structure of the form (\ref{Ex:AlmNorCon}) for some normal almost contact 
structure $(\Phi,\xi,\eta)$.
\end{lemma}

\begin{proof}
Without loss of generality, we can assume $L = \mathbbm R_M$, so that $DL \cong TM \oplus \mathbb R_M$. 
It is clear that, under this identification, $\varphi$ is necessarily of the form
	\begin{equation}\label{eq:4tuple}
	\varphi(X,r)=(\Phi(X)-r\xi,\eta(X)+ gr)
	\end{equation}
for some quadruple $(\Phi, \xi, \eta, g)$, where $\Phi$ is a $(1,1)$-tensor, $\xi$ is a vector field, $\eta$ is a $1$-form, and $g$ is a smooth function on $M$. Locally, we can achieve $g = 0$ as follows. First of all, let $f \in C^\infty (M)$. A straightforward computation shows that, under the line bundle automorphism $\mathbbm R_M \to \mathbbm R_M$, $(x, r) \mapsto e^{-f(x)} r$, the quadruple $(\Phi, \xi, \eta, g)$ changes into
\[
(\Phi +df\otimes \xi,\  \xi ,\ \eta + df\circ \Phi+(\xi(f) -g)df,\  g-\xi(f))
\]
Now, from $\varphi^2 = - \mathrm{id}$, we easily find that $\xi$ is everywhere non-zero. Hence, locally, around every point, there exists a function $f$ such that $\xi (f) = g$. This concludes the proof.
\end{proof}

\begin{remark}
Not all integrable complex structures on $DL$ are globally of the form (\ref{Ex:AlmNorCon}), in general, not even when $L = \mathbbm R_M$ is the trivial line bundle. To see this, let $M$ be a manifold such that $\mathrm{H}_{\mathrm{dR}}^1(M)
\neq 0$, and let $(\Phi',\xi',\eta')$ be a normal almost contact structure on $M$ (such manifolds exist, and the $1$-dimensional sphere provides the simplest possible example). Now, pick a closed, but not exact, $1$-form $\alpha$ on $M$, and put
	\begin{equation}
	\Phi = \Phi' +\alpha\otimes \xi', \quad \xi = \xi' ,\quad \eta =  \eta'+\alpha\circ\Phi' +\alpha(\xi')\alpha, \quad g = -\alpha(\xi').
	\end{equation}
Then the endomorphism $\varphi : D \mathbbm R_M \to D \mathbbm R_M$ given by (\ref{eq:4tuple}) is an integrable complex structure that cannot be put in the form (\ref{Ex:AlmNorCon}) by a global line bundle automorphism $\mathbbm R_M \to \mathbbm R_M$.
\end{remark}

\subsection{Local normal form}\label{subsec:nn}

\begin{theorem}\label{theor:ANN}
Let $L \to M$ be a line bundle equipped with a complex structure $\varphi : DL \to DL$ on the gauge algebroid. Then, locally, around every point of $M$, there are
\begin{itemize}
\item[$\triangleright$] coordinates $(u, x^1, \ldots, x^n, y^1, \ldots, y^n)$ on M, and
\item[$\triangleright$] a flat connection $\nabla$ in $L$, such that
\end{itemize}
\begin{equation}\label{eq:varphi_can}
\varphi \mathbb 1 = \nabla_{\partial / \partial u}, \quad \text{and} \quad \varphi \nabla_{\partial / \partial x^i} = \nabla_{\partial / \partial y^i}.
\end{equation}
In other words, locally, around every point of $M$, there is trivialization $L \cong \mathbbm R_{ \mathbbm R \times  \mathbbm C^n}$ identifying $\varphi$ with $\varphi_{\mathit{can}}$ from Example \ref{ex:varphi_can}.
\end{theorem}

The proof will essentially follow from the Newlander-Nirenberg theorem after applying the \emph{homogenization trick} \cite{VW2017} which we now recall. First of all, consider the frame bundle $\widetilde M =  L^\ast \smallsetminus 0 \to M$ of $L$: it is a principal $\mathbbm R^\times$-bundle. We denote by $h : \mathbbm R^\times \times \widetilde M \to \widetilde M$, $(s, \varepsilon) \mapsto h_s (\varepsilon)$ the group action, and by $\mathcal E = \frac{d}{dt}|_{t = 0} h_{\mathrm{exp} (t)}$ the restriction to $\widetilde M$ of the Euler vector field. A section $\lambda$ of $L$ corresponds to a linear function on $L^\ast$, and, by restriction, to a \emph{homogeneous function} $\widetilde \lambda$ on $\widetilde M$, where, by ``homogeneous'', we mean that $h^\ast_s (\widetilde \lambda) = s \widetilde \lambda$, for all $s \in \mathbbm R^\times$. In particular, $\mathcal E (\widetilde \lambda) = \widetilde \lambda$. Every homogeneous function on $\widetilde M$ arises in this way. Secondly, let $\Delta$ be a derivation of $L$. Then there exists a unique vector field $\widetilde \Delta$ on $\widetilde M$ such that
\[
\widetilde \Delta (\widetilde \lambda) = \widetilde{\Delta \lambda}, \quad \text{for all $\lambda \in \Gamma (L)$}.
\]
Vector field $\widetilde \Delta$ is homogeneous in the sense that $h_s^\ast (\widetilde \Delta) = \widetilde \Delta$, for all $s \in \mathbbm R^\times$. In particular, $\widetilde \Delta$ commutes with $\mathcal E$ and it is projectable onto $M$ with projection $\sigma (\Delta)$. Every homogeneous vector field on $\widetilde M$ arises in this way. Notice that $\widetilde{\mathbb 1} = \mathcal E$. Thirdly, let $\varphi : DL \to DL$ be a vector bundle endomorphism. Then there exists a unique $(1,1)$-tensor $\widetilde \varphi : T \widetilde M \to T \widetilde M$ such that 
\begin{equation}\label{eq:varphi_tilde}
\widetilde \varphi \widetilde \Delta = \widetilde{\varphi \Delta}, \quad \text{for all $\Delta \in \Gamma (DL)$}.
\end{equation}
The $(1,1)$-tensor $\widetilde \varphi$ is homogeneous in the sense that $h^\ast_s (\widetilde \varphi) = \widetilde \varphi$, for all $s \in \mathbbm R^\times$. In particular, the Lie derivative $\mathcal L_{\mathcal E} \widetilde \varphi$ vanishes. Every homogeneous $(1,1)$-tensor on $\widetilde M$ arises in this way. Additionally, $\varphi$ is an integrable complex structure if and only if $\widetilde \varphi$ is a complex structure on $\widetilde M$.

\begin{example}
An immediate consequence of the homogenization construction described above is that all odd dimensional real projective spaces possess a(n 
integrable) complex structure on the gauge algebroid of the dual of their tautological bundle. Indeed, let $k$ be a positive integer, and 
let $L$ be the dual of the tautological bundle on the projective space $\mathbbm{RP}^{2k-1}$. The total space of the frame bundle of $L$ 
identifies canonically with $\mathbbm R^{2k} \smallsetminus \{0\}$, and the action $h$ of $\mathbbm R^\times$ consists of homotheties. The 
standard complex structure on $\mathbbm R^{2k} = \mathbbm C^k$ is homogeneous, hence $DL$ is equipped with an integrable 
complex structure. Viewing the sphere as a double cover of the projective space, we also conclude that there is a canonical integrable complex structure on the Atiyah algebroid of the trivial line bundle over any odd dimensional sphere, and one can show that this complex structure does actually correspond to the normal almost contact structure underlying the well-known canonical Sasaki structure.
\end{example}

\begin{proof}[Proof of Theorem \ref{theor:ANN}.]

Let $\varphi : DL \to DL$ be an integrable complex structure. Consider $\widetilde \varphi$. It is a complex structure on $\widetilde M$. As $\mathcal E$ is nowhere vanishing, it can be locally completed to a holonomic complex frame, i.e.~locally, around every point of $\widetilde M$, there are coordinates $(T, U, X^1, \ldots, X^n, Y^1, \ldots, Y^n)$ such that
\[
\mathcal E = \frac{\partial}{\partial T}, \quad \widetilde \varphi \mathcal E = \frac{\partial}{\partial U}, \quad \text{and} \quad \widetilde \varphi \frac{\partial}{\partial X^i} = \frac{\partial}{\partial Y^i}.
\]
As all coordinate vector fields commute with $\mathcal E$, they all come from (commuting) derivations of $L$. In particular
\begin{itemize}
\item[$\triangleright$] $(U, X^1, \ldots, X^n, Y^1, \ldots, Y^n)$, are pull-backs via the projection $\widetilde M \to M$ of uniquely defined coordinates $(u, x^1, \ldots, x^n, y^1, \ldots, y^n)$ on $M$, and
\item[$\triangleright$] there exists a unique flat connection $\nabla$ in $L$ such that 
\[
\frac{\partial}{\partial U} = \nabla_{\partial / \partial u}, \ \ldots\ , \ \frac{\partial}{\partial X^i} = \nabla_{\partial / \partial x^i}, \ \ldots\ ,\  \frac{\partial}{\partial Y^i} = \nabla_{\partial / \partial y^i},\ \ldots
\]
\end{itemize}
From (\ref{eq:varphi_tilde}), the coordinates $(u, x^1, \ldots, x^n, y^1, \ldots, y^n)$ on $M$ and the flat connection $\nabla$ possess all the required properties.
\end{proof}

As an immediate corollary of Theorem \ref{theor:ANN} and Lemma \ref{lem:ac} we get a local normal form for normal almost contact structures.

\begin{corollary}
Let $(\Phi,\xi,\eta)$ be a normal almost contact structure on a manifold $M$. Then, around every point, there exist local coordinates 
$(u,x^i,y^i)$ and a local function $f$, such that: 
	\begin{enumerate}
	\item $\xi=\frac{\partial }{\partial u}$,
	\item $\eta=du+\frac{\partial f}{\partial y^i}dx^i- \frac{\partial f}{\partial x^i}dy^i$,
	\item $\Phi= dx^i\otimes \frac{\partial}{\partial y^i}- dy^i\otimes \frac{\partial}{\partial x^i} + df \otimes \frac{\partial }{\partial u}$.
	\end{enumerate}
\end{corollary}

\subsection{Dolbeault-Atiyah cohomology}
Let $L \to M$ be a line bundle, and let $\varphi : DL \to DL$ be an integrable complex structure on the gauge algebroid of $L$. Similarly as in the case of a complex manifold, there is a cohomology theory attached to $\varphi$. Namely, consider the complexification $DL \otimes \mathbbm C$ of the gauge algebroid and denote by $D^{(1,0)} L$ and $D^{(0,1)} L$ the $+\mathrm{i}$ and the $-\mathrm{i}$-eigenbundles of $\varphi$ respectively, so that 
\[
DL \otimes \mathbbm C = D^{(1,0)} L \oplus D^{(0,1)} L,
\]
and complex Atiyah forms $\Omega_L^\bullet \otimes \mathbbm C$ splits as
\[
\Omega_L^\bullet \otimes \mathbbm C = \bigoplus_{r,s} \Omega^{(r,s)}_L,
\]
where we denoted by $\Omega^{(r,s)}_L$ the sections of the (complex) vector bundle
\[
\wedge^r (D^{(1,0)}L)^\ast \otimes \wedge^s  (D^{(0,1)} L)^\ast \otimes L.
\]
The de Rham differential $d_D$ splits, in the obvious way, as $d_D = \partial_D + \overline \partial_D$, where 
\[
\partial_D : \Omega^{(\bullet, \bullet)}_L \to \Omega^{(\bullet + 1, \bullet)}_L, \quad \text{and} \quad \overline \partial_D : \Omega^{(\bullet, \bullet)}_L \to \Omega^{(\bullet, \bullet + 1)}_L,
\]
and the integrability of $\varphi$ is equivalent to
\[
\partial_D^2 = \overline \partial{}^2_D = \partial_D \overline \partial_D + \overline \partial_D \partial_D = 0.
\]
We call the cohomology of $\overline{\partial}_D$ the \emph{Dolbeault-Atiyah cohomology}.

\begin{theorem}\label{theor:DAC}
The Dolbeault-Atiyah cohomology vanishes locally.
\end{theorem}

\begin{proof}
In view of Theorem \ref{theor:ANN}, it is enough to work in the case when $M = \mathbbm R \times \mathbbm C^n$. Let $u$ be the standard (real) coordinate on the first factor and let $z^i = x^i + \mathrm{i} y^i$, $i = 1, \ldots, n$, be the standard complex coordinates on the second factor. We can also assume that $L = \mathbbm R_M$ is the trivial line bundle and (\ref{eq:varphi_can}) holds with $\nabla$ being the canonical flat connection on $\mathbbm R_M$. In this case $D^{(1,0)} L$ is spanned by the complex derivations
\begin{equation}\label{eq:D^(1,0)_gen}
\square := \frac{1}{2} \left(\mathbb 1 - \mathrm{i} \nabla_{\partial / \partial u} \right), \quad \text{and} \quad \nabla_i =  \frac{1}{2} \left(\nabla_{\partial/\partial x^i} - \mathrm{i} \nabla_{\partial / \partial y^i} \right), \quad i = 1, \ldots, n.
\end{equation}
It is easy to see that every complex Atiyah form $\omega$ on $\mathbbm R_M$ can be uniquely written as
\[
\omega = \omega_0 + \omega_1 \wedge \mathfrak k
\]
where $\omega_0, \omega_1$ are standard complex forms on $M$ and 
\[
\mathfrak k = \mathfrak j + \mathrm{i} \cdot du.
\]
A long but straightforward computation exploiting
(\ref{eq:D^(1,0)_gen}), shows that 
\[
\overline \partial_D \omega = \overline \partial \omega_0 + \left(\overline \partial \omega_1+(-)^{|\omega_0|} \left(\omega_0 + \mathcal L_Y \omega_0 \right)\right) \wedge \mathfrak k
\]
where
\[
Y := \sigma (\overline \square) = \frac{\mathrm{i}}{2} \frac{\partial}{\partial u},
\]
and $\overline \partial$ is the standard Dolbeault differential on $\mathbbm C^n$ (acting on forms on $\mathbbm R \times \mathbbm C^n$ in the obvious way). So $\omega$ is $\overline \partial_D$-closed iff
\[
\overline \partial \omega_0 =  \overline \partial \omega_1+(-)^{|\omega_0|} \left(\omega_0 + \mathcal L_Y \omega_0 \right) = 0.
\]
In this case, use the vanishing of standard Dolbeault cohomology (with a real parameter $u$), to choose a form $\rho_0$ such that $\overline \partial \rho_0 = \omega_0$. As the Lie derivative along $Y$ commutes with $\overline \partial$ we find
\[
\overline \partial \left( \omega_1 - (-)^{|\rho_0|} \left(\rho_0 + \mathcal L_Y \rho_0 \right)\right) = 0,
\]
and we can choose $\rho_1$ such that $\overline \partial \rho_1 = \omega_1 - (-)^{|\rho_0|} \left(\rho_0 + \mathcal L_Y \rho_0 \right)$. It is now easy to see that
\[
\overline \partial_D \left (\rho_0 + \rho_1 \wedge \mathfrak k \right) = \omega.
\]
This concludes the proof.
\end{proof}

\begin{remark}\label{rem:Dolbeault}
It immediately follows from Theorem \ref{theor:DAC} that the cohomology of $\partial_D$ does also vanish locally.
\end{remark}

\end{document}